\documentclass[11pt,twoside, a4paper, english, reqno]{amsart}
\usepackage{graphicx, amsmath, varioref, amscd, amssymb,color, bm, stmaryrd, amsthm}
\usepackage{fancybox}
\usepackage[pdftex, colorlinks=true, backref]{hyperref}

\numberwithin{equation}{section}
\allowdisplaybreaks

\newtheorem{theorem}{Theorem}[section]
\newtheorem{lemma}[theorem]{Lemma}
\newtheorem{corollary}[theorem]{Corollary}

\newtheorem{proposition}[theorem]{Proposition}
\theoremstyle{definition}
\newtheorem{definition}[theorem]{Definition}

\theoremstyle{remark}

\newcommand{\Grad}{\nabla}
\newcommand{\vr}{\varrho}

\newcommand{\weak}{\rightharpoonup}

\newcommand{\weakstar}{\overset{\star}\rightharpoonup}

\newcommand{\inb}{\in_{\text{b}}}

\newcommand{\R}{\mathbb{R}}

\newcommand{\Om}{\ensuremath{0,L}}

\newcommand{\Dom}{(0,T)\times\Omega}

\newcommand{\eps}{\epsilon}

\newcommand{\up}{\operatorname{Up}}

\newcommand{\equationbox}[2]{\begin{center}

\fbox{
\begin{minipage}{0.96\textwidth}
#1
\end{minipage}
}
\end{center}
}

\begin{document}

\title[Convergent finite differences for 1D viscous isentropic flow]{Convergent finite differences \\for 1D viscous isentropic flow \\ in Eulerian coordinates} 

\author[Karper]{Trygve K. Karper}\thanks{The research is funded by the Research Council of Norway (proj. 205738).
This research was conducted at Paul Sabatier University, Toulouse. The author 
wishes to express his gratitude towards Paul Sabatier University for its hospitality.}

\address[Karper]{\newline
Hammerstads gt. 20, 0363 Oslo, Norway.
}
\email[]{\href{karper@gmail.com}{karper@gmail.com}}
\urladdr{\href{http://folk.uio.no/~trygvekk}{folk.uio.no/\~{}trygvekk}}

\date{\today}



\maketitle
\begin{abstract}
We construct a new finite difference method 
for the flow of ideal viscous isentropic gas  in one spatial dimension. 
For the continuity equation, the method is a standard 
upwind discretization. For the momentum equation, 
the method is an uncommon upwind discretization, 
where the moment and the velocity are solved on dual grids.
Our main result is convergence of the method as discretization 
parameters go to zero. Convergence is proved 
by adapting the mathematical existence theory of Lions 
and Feireisl to the numerical setting. 
\end{abstract}

\setcounter{tocdepth}{1}
\tableofcontents

\section{Introduction}
In this paper, we will develop a convergent finite difference method 
for the flow of an ideal viscous isentropic gas in one spatial dimension.
We will assume that the flow may be modeled by the Navier-Stokes system (cf.~\cite{Serrin}):
\begin{align}
	\vr_t + (\vr u)_x &= 0, & \text{in $\R^+ \times (0,L)$}, \label{eq:cont}\\
	(\vr u)_t + (\vr u^2)_x &= \mu u_{xx} - p(\vr)_x,& \text{in $\R^+ \times (0,L)$}. \label{eq:moment}
\end{align}
The unknowns in this system are the fluid density $\vr=\vr(t,x)$ and the fluid velocity $u=u(t,x)$.
For an isentropic flow, the ideal pressure law takes the form
\begin{equation*}
	p = a\vr^\gamma,  \quad a>0,
\end{equation*}
where the value of $\gamma$ is determined by the specific gas in question.
In this paper, we will be forced require that
\begin{equation*}\label{eq:gamma}
	\frac{3}{2} < \gamma < 2,
\end{equation*}
to prove convergence of the method.
Note that this significantly limits 
the physical applicability of the convergence result. While monoatomic 
gases ($\gamma \sim \frac{5}{3}$), such as helium, are included,
diatomic gases ($\gamma \sim \frac{7}{5}$), such as air, are not.
The condition is necessary for convergence, but not stability, 
of the method.

At the boundary, 
 \eqref{eq:cont} - \eqref{eq:moment} is augmented with no-slip conditions,
\begin{equation*}\label{eq:bc}
	u(t, 0) = u(t, L) = 0.
\end{equation*}
and the initial conditions,
\begin{equation}\label{eq:init}
	0 < \vr(0,x) = \vr_0(x) \in L^{\gamma}(\Om), \qquad u(0, x) = u_0(x) \in L^\infty(\Om).
\end{equation}

Existence of global classical solutions for the system \eqref{eq:cont}-\eqref{eq:moment}
is well-known and was first established by Kanel' \cite{Kanel} for smooth data (see also \cite{Kazhikov}). 
For non-smooth initial data, corresponding results have been obtained 
by David Hoff \cite{Chen:2000yq, Hoff}. These works are all based 
on the same approach: find pointwise upper and lower bounds 
on the density, then use these bounds to   
estimate the second derivative of the velocity. With smooth initial density, 
the density remain  $H^1$, while with discontinuous 
initial density, the density is at most $BV$ as 
initial discontinuities persists for all time (see \cite{Hoff}).
The pointwise bounds
on the density are obtained by tracking certain quantities 
along streamlines rendering the entire existence theory 
essentially Lagrangian.

In the literature, one can find a huge variety of numerical 
methods appropriate for \eqref{eq:cont}-\eqref{eq:moment}.
However, very few of these methods  have been proven to converge.
This lack of rigorous results are most likely a consequence 
of the Lagrangian nature of the existence theory. 
In fact, prior to this paper, all convergent methods 
have been discretized in Lagrangian coordinates 
and are due to David Hoff and collaborators \cite{Zarnowski:1991uq, 
Zhao:1994fk, Zhao:1997qy}.
That being said, there are also several existence results that 
utilize discrete Lagrangian approximation schemes (e.g.~\cite{Chen:2000yq,Hoff,  Jenssen}) 
to construct solutions. 
For practical applications, results in Eulerian 
coordinates is often desirable and
a mapping from discrete Lagrange to discrete Euler is costly. 
For this reason, most practitioners would employ an Eulerian 
discretization. 

In this paper, we will discretize 
the equation in Eulerian coordinates. As a consequence, 
obtaining pointwise bounds on the density becomes
highly involved and will not be pursued in this paper.
Instead, we will develop a convergence theory 
in the spirit of the continuous existence theory \cite{Feireisl, Lions:1998ga, Straskraba}
as sparked by P.~L.~Lions. In particular, we will not 
obtain any form of continuity of the discrete density 
and instead prove strong convergence of the density 
using renormalization and what is known 
as the effective viscous flux (cf.~\cite{Straskraba}).

In more than one spatial dimension, the literature 
is almost void of convergent numerical methods. 
For the full system, the only result is the paper \cite{Karper},
by the author, in which a convergent finite element discretization of \eqref{eq:cont}-\eqref{eq:moment}
is developed. Over the last years, there have 
also been developed a series of convergent methods
\cite{Eymard, Gallouet2, Gallouet1, Karlsen1, Karlsen2,Karlsen3}
for the Stokes version of \eqref{eq:cont}-\eqref{eq:moment}.
The method we shall develop and analyze in this paper can
be considered as an archtype for all the methods 
mentioned above. That is, in one spatial dimension, all 
of these methods are in a sense equivalent and of the form we shall consider here. 
In addition, the method we present here is the first 
 finite difference method for which 
convergence is proved.

The paper is organized as follows:	
In the next section, we will define the numerical 
method and state the main convergence result. 
Then, we will derive stability of the method 
and some other basic properties. 
In Section 4, we will derive an equation 
for the effective viscous flux and also 
provide a higher integrability estimate for 
the density. Section 5 concerns 
passing to the limit in the method. 
In particular, we prove that the limit 
almost satisfies \eqref{eq:cont}-\eqref{eq:moment}. 
The remaining ingredient is to prove strong convergence 
of the density which is proved in the final section.

\section{The numerical method and main result}

Our method will be posed on a uniform grid in both space and time. 
In time, we shall approximate at discrete points 
$t^k= k\Delta t$, where $\Delta t$ is assumed to be of the order $\Delta x$.
In space, we will divide the domain $[0,L]$ into $N$
intervals of length $\Delta x = L/N$. It will be convenient 
to write this as
\begin{equation*}
	[0,L] = \bigcup_i [x_{i-1/2}, x_{i+1/2}], \quad i=0, \ldots, N-1,
\end{equation*}
where we have introduced the slightly confusing notation
\begin{equation*}
	x_{i-1/2} = i\Delta x,\quad i =0, \ldots, N.
\end{equation*}
We shall also need the dual grid given by the midpoint nodes
$$
x_i = \left(i+\frac{1}{2}\right)\Delta x, \quad i=0, \ldots, N-1.
$$
The numerical method will approximate the density on 
the dual grid and the velocity on the standard grid:
\begin{equation*}
	\begin{split}
	 \vr(k\Delta t, x_i) &\approx \vr_i^k, \qquad i=0, \ldots, N-1,\\
	 u(k\Delta t, x_{i-1/2}) &\approx u_{i-1/2}^k, \qquad i = 0, \ldots, N.		
	\end{split}
\end{equation*}
Staggered grid of this kind is necessary for incompressible flows, and
as a consequence also widely used to develop all speed compressible flow methods.

\subsection{Discrete operators}
To discretize the convective terms, we shall utilize an upwind 
method. To facilitate this, we introduce the notation
\begin{equation*}
	u^+ = \max\{u, 0\}, \quad u^- = \min\{u,0\}.
\end{equation*}
We shall also need the average velocity over an interval
\begin{equation*}
	\widehat u_i = \frac{u_{i-1/2} + u_{i+1/2}}{2}.
\end{equation*}
For the continuity equation \eqref{eq:cont}, we shall use the 
following upwind flux 
\begin{equation*}
	\up(\vr u)_{i+1/2} = \vr_i u_{i+1/2}^+ + \vr_{i+1} u_{i+1/2}^-.
\end{equation*}
This upwind flux will also be used for the momentum equation,
where we upwind the (averaged) momentum 
\begin{equation*}
	\begin{split}
		\up(\vr\widehat u u )_{i+1/2} & = (\vr_i\widehat u_i) u_{i+1/2}^+ + (\vr_{i+1}\widehat u_{i+1}) u_{i+1/2}^-.
	\end{split}
\end{equation*}

The remaining derivatives will be discretized using the operators
\begin{equation*}
	\begin{split}
		\partial_{i+1/2}f = \frac{f_{i+1} - f_i}{\Delta x}, \quad \partial_i v = \frac{v_{i+1/2} -v_{i-1/2}}{\Delta x},
	\end{split}
\end{equation*}
which defines the following (standard) Laplace operator
\begin{equation*}
	\Delta_{i+1/2}u_h = \partial_{i+1/2}\partial_{i} u_h = \frac{u_{i-1/2} - 2u_{i+1/2} + u_{i+3/2}}{\Delta x^2}.
\end{equation*}
For time discretization, we will use implicit time-stepping
$$
	\partial_t^k f_i = \frac{f_i^k- f_{i}^{k-1}}{\Delta t}.
$$

\subsection{Numerical method}
The numerical method is defined as follows:
\vspace{0.2cm}

\equationbox
{
\begin{definition}\label{def:scheme}
Given $\Delta t$, $\Delta x < 1$, and initial data \eqref{eq:init}, 
define the numerical initial data
\begin{equation*}
	\vr^0_i = \frac{1}{\Delta x}\int_{x_{i-1/2}}^{x_{i+1/2}}\vr_0(y)~dy, \quad u^0_{i-1/2} = u_0(x_{i-1/2}), 
	\qquad i=0, \ldots, N-1.
\end{equation*}
Determine sequentially the numbers
$$
	(\vr_i^k, u_i^k), \quad i=0, \ldots, N-1, \quad k=0, \ldots, M,
$$
 solving the nonlinear system
\begin{align}
	\partial_t^k\vr_i + \partial_i \up(\vr^k u^{k}) &= 0 \label{num:cont}\\
	\partial_t^k\left(\frac{\vr_i \widehat u_i + \vr_{i+1} \widehat u_{i+1}}{2}\right) 
		&+ \left(\frac{\up(\vr^k \widehat u^k u^{k})_{i+3/2} -\up(\vr^k \widehat u^k u^{k})_{i-1/2}}{2\Delta x}\right) \nonumber \\
		&= \mu \Delta_{i+1/2}u^{k}_h - \, \partial_{i+1/2}p(\vr^{k}_h), \label{num:moment} \\
	u_{-1/2}^k &= u_{N-1/2}^k = 0 \nonumber.
\end{align}
\end{definition}

}

\vspace{0.2cm}

It is not completely trivial that our 
numerical method is well-defined. 
Since the system of equations \eqref{num:cont}-\eqref{num:moment}
are both nonlinear and implicit, existence of 
a solutions needs to be established.

\begin{lemma}\label{lem:}
Let $0 < \Delta t, \Delta x < 1$ be fixed. 
The nonlinear implicit system of equations \eqref{num:cont}-\eqref{num:moment}
admits at least one solution.
\end{lemma}
\begin{proof}
The existence can be proved using 
a toplogical degree argument.
Since the proof is very similar to 
the corresponding result in \cite{Gallouet:2006lr, Karlsen2,
Karper}, we do not give the details here and 
instead refer the reader to any of the mentioned papers.
\end{proof}

From \cite{Gallouet:2006lr} and \cite{Karlsen2}, 
we also have that the method 
preserves strict positivity of the density:
\begin{lemma}\label{lem:}
Let $\{(\vr^k_i, u_{i-1/2}^k)\}_{i}$ satisfy 
the continuity scheme \eqref{num:cont}. 
Then, 
\begin{equation*}
	\min_i \vr_i^k \geq \min_i \vr^{k-1}_i\left(\frac{1}{1+\Delta t\max_i |u_{i+1/2}^k|}\right).
\end{equation*}

\end{lemma}

\subsection{Extension}
To analyze the finite difference method, 
it will be of great convenience to 
extend the numerical solution 
to  all of $[0,T)\times (0,L)$.
Since the density appears nonlinearly 
in the pressure, we will use piecewise 
constants to extend it:
\begin{equation}\label{def:ext1}
	\begin{split}
		\vr_h(t,x)
		&= \vr_i^k, \quad \forall (t,x) \in [t^{k},t^{k+1})
		\times [x_{i-1/2},x_{i+1/2}).
	\end{split}
\end{equation}

For the velocity, we shall utilize piecewise 
continuous linears in space and piecewise 
constants in time.
\begin{equation}\label{def:ext2}
	u_h(t,x)
	= u_{i-1/2}^k 
	+ \left(\frac{x-x_{i-1/2}}{\Delta x}\right)
	\left(u^k_{i+1/2} -u^k_{i-1/2} \right),
\end{equation}
for all $(t,x) \in [t^k, t^{k+1})\times [x_{i-1/2},x_{i+1/2})$.
Note that an extension of the average velocity $\widehat u_i^k$ 
is now given by the $L^2$ projection onto piecewise constant:
\begin{equation*}
	\widehat u_h(t,x) = \Pi_h^Q [u_h]
	= \frac{1}{\Delta x}\int_{x_{i-1/2}}^{x_{i+1/2}}u_h~ dx,
	\quad x \in (x_{i-1/2}, x_{i+1/2}).
\end{equation*}

\subsection{Main result}
The following theorem is our main result 
in this paper. 
\begin{theorem}\label{thm:main}
Assume  we are given initial data $(\vr_0, u_0)$ satisfying \eqref{eq:init} 
and a finite final time $T>0$. Let $\{(\vr_h, u_h)\}_{h>0}$ 
be a family of numerical solutions, constructed through Definition \ref{def:scheme}
and \eqref{def:ext1}-\eqref{def:ext2}, with
$$
	\Delta t= \Delta x= h.
$$
Then, as $h \rightarrow 0$, $u_h \weak u$ in $L^2(0,T;W^{1,2}_0(\Om))$, 
$\vr_h \rightarrow \vr$ a.e on $(0,T)\times (0,L)$, where $(\vr, u)$
is a weak solution of \eqref{eq:cont}-\eqref{eq:moment}:
\begin{align*}
	\int_0^T\int_0^L \vr(\phi_t + u\cdot \phi_x)~dxdt &= \int_0^L \vr_0 \phi(0, \cdot)~ dx, \\
	\int_0^T\int_0^L (\vr u)v_t + \vr u^2 v_x + \mu u_x v_x - p(\vr)v_x~dxdt&=
	\int_0^L \vr_0 u_0 \phi(0,\cdot)~dx,
\end{align*}
for all $(\phi, v) \in C_0^\infty([0,T)\times (0,L))$.
\end{theorem}
Theorem \ref{thm:main} will follow as 
a consequence of the various results 
stated and proved in the upcoming sections. The proof will 
be completed in Section 6.

\section{Stability and energy estimates}
In this section, we we will derive a 
numerical analog of the continuous energy estimate.
This estimate yields in particular stability of the method, 
but it will also provides us with necessary $L^p$ bounds 
uniform in the discretization parameters.

\subsection{Rernormalized continuity scheme}

We shall need the following renormalized continuity 
scheme at several occasions. The term renormalized 
is motivated by the corresponding continuous equation 
and its role in the existence theory (cf.~\cite{Lions:1998ga}).
\begin{lemma}\label{lem:renorm}
Let $B\in C^1(0,R^+)$ and define $b(z) = zB'(z)- B(z)$. 
If $\{(\vr_i^k, u_{i+1/2}^k\}_{i,k}$ satisfies the continuity scheme \eqref{num:cont}, 
then the following identity holds
	\begin{equation}\label{num:renorm}
		\begin{split}
			 &\partial^k_t B(\vr_i)	+ \partial_i \up\left(B(\vr^k)u^k\right) + b(\vr^k_i)\partial_i u^k + \mathcal{P}\left[B(\vr^k_i),u^k\right] = 0,
		\end{split}
	\end{equation}
	where $\mathcal{P}$ is given by
	\begin{equation*}
		\begin{split}
			\mathcal{P}\left[B(\vr^k_i),u\right] &= \Delta tB''(\vr_i^*)|\partial_t^k \vr_i|^2 \\
			&\quad -B''(\vr^\dagger_i)[\vr_{i+1} - \vr_i]^2u_{i+1/2}^- + B''(\vr_i^\ddagger)[\vr_i - \vr_{i-1}]^2 u_{i-1/2}^+.
		\end{split}
	\end{equation*}
	Here, $\vr_i^*$, $\vr_i^\dagger$, and $\vr_i^\ddagger$, are some numbers in the range 
	$[\vr_i^{k-1},\vr_i^{k-1}]$, $[\vr^k_i, \vr^k_{i+1}]$, and $[\vr^k_{i-1}, \vr^k_{i}]$, respectively.
\end{lemma}
\begin{proof}
We begin by multiplying \eqref{num:cont} with $B'(\vr)$ to obtain
\begin{equation*}
	B'(\vr_i)\partial^k_t \vr_i = -B'(\vr_i)\partial_i \up(\vr u).
\end{equation*}	
By applying Taylor expansion, we find that
\begin{equation}\label{eq:renorm1}
	\begin{split}
		\partial_t^k B(\vr) + \Delta tB''(\vr^*)|\partial_t^k \vr_i|^2 = -B'(\vr_i)\partial_i \up(\vr u), 
	\end{split}
\end{equation}
where $\vr^*$ is some number in $[\vr_i^{k-1}, \vr_i^k]$.
	
By adding and subtracting and applying another Taylor expansion,
\begin{equation}\label{eq:renorm2}
	\begin{split}
		&(\Delta x) B'(\vr_i)\partial_i \up(\vr u) \\
		&\qquad = B'(\vr_i)\left[\vr_i u_{i+1/2}^+ + \vr_{i+1} u_{i+1/2}^- -\vr_{i-1} u_{i-1/2}^+ - \vr_{i} u_{i-1/2}^-\right] \\
		&\qquad = B'(\vr_i)\vr_i ( u_{i+1/2} -  u_{i-1/2}) + B'(\vr_i)[\vr_{i+1} -\vr_i]u_{i+1/2}^- \\
		&\qquad \quad +B'(\vr_i)[\vr_i - \vr_{i-1}]u_{i-1/2}^+ \\
		&\qquad = B'(\vr_i)\vr_i ( u_{i+1/2} -  u_{i-1/2}) + [B(\vr_{i+1})- B(\vr_i)]u_{i+1/2}^- \\
		&\qquad \quad + \left[B(\vr_i) - B(\vr_{i-1})\right]u_{i-1/2}^+ + B''(\vr^\dagger_i)[\vr_{i+1} - \vr_i]^2u_{i+1/2}^- \\
		&\qquad \quad - B''(\vr_i^\ddagger)[\vr_i - \vr_{i-1}]^2 u_{i-1/2}^+ \\
		&\qquad = (\Delta x) \left(B'(\vr_i)\vr_i \partial_i u + \partial_{i+1/2}B(\vr)u_{i+1/2}^- + \partial_{i-1/2}B(\vr)u_{i-1/2}^+\right) \\
		&\qquad \quad + B''(\vr^\dagger_i)[\vr_{i+1} - \vr_i]^2u_{i+1/2}^- - B''(\vr_i^\ddagger)[\vr_i - \vr_{i-1}]^2 u_{i-1/2}^+
	\end{split}
\end{equation}
Here, $\vr^\dagger_i$ and $\vr^\ddagger_i$
are some numbers in $[\vr_i, \vr_{i+1}]$ and $[\vr_i, \vr_{i-1}]$, respectively. 
Next, we calculate
\begin{equation}\label{eq:renorm3}
	\begin{split}
		&B'(\vr_i)\vr_i \partial_i u + \partial_{i+1/2}B(\vr)u_{i+1/2}^- + \partial_{i-1/2}B(\vr)u_{i-1/2}^+	\\
		&= b(\vr_i)\partial_i u + B(\vr_i)\partial_i u+ \partial_{i+1/2}B(\vr)u_{i+1/2}^- + \partial_{i-1/2}B(\vr)u_{i-1/2}^+ \\
		&= b(\vr_i)\partial_i u + \partial_i \up(B(\vr)u)
\end{split}
\end{equation}
We conclude the proof by combining \eqref{eq:renorm3} in \eqref{eq:renorm2} and \eqref{eq:renorm1}.
\end{proof}

\subsection{The convection operator}
To prove stability of the method, we shall 
multiply the momentum scheme \eqref{num:moment} by 
$u_{i+1/2}^k$ and sum over all $i$. For this purpose, 
we shall need the following identity for the convective 
discretization.

\begin{lemma}\label{lem:conv}
The following identity holds
\begin{equation*}
	\begin{split}
		&\Delta x \sum_{i=0}^{N-1}\left(\frac{\up(\vr^k \widehat u^k u^{k})_{i+3/2} 
		-\up(\vr^k \widehat u^k u^{k})_{i-1/2}}{2\Delta x}\right)u_{i+1/2}^k \\
		&\qquad = -\Delta x \sum_{i=0}^{N-1}\up\left(\vr^k u^{k}\right)_{i+1/2}\partial_{i+1/2}~\left(\frac{|\widehat u^k|^2}{2}\right)
		+ \mathcal{N}_2 \\
		&\qquad = -\Delta x\sum_i \partial_t^k \vr_i\left(\frac{|\widehat u^k|^2}{2}\right)+ \mathcal{N}_2.
	\end{split}
\end{equation*}	
where the numerical diffusion term $\mathcal{N}_1$ is given by
\begin{equation*}
	\mathcal{N}_2 = \frac{(\Delta x)^2 }{2}\sum_{i=0}^{N-1}\left|\up\left(\vr^k u^{k}\right)_{i+1/2}\right|\left|\partial_{i+1/2} \widehat u^k\right|^2.
\end{equation*}
\end{lemma}
\begin{proof}
	By applying summation by parts, we see that
	\begin{equation}\label{conv:1}
		\begin{split}
		&\Delta x \sum_{i=0}^{N-1}\left(\frac{\up(\vr^k \widehat u^k u^{k})_{i+3/2} -\up(\vr^k \widehat u^k u^{k})_{i-1/2}}{2\Delta x}\right)u_{i+1/2}^k	 \\
		&\qquad = -\Delta x \sum_{i=0}^{N-1}\up(\vr^k \widehat u^k u^{k})_{i+1/2}\left(\frac{u_{i+3/2}^k - u_{i-1/2}^k}{2\Delta x}\right) \\
		&\qquad = -\Delta x \sum_{i=0}^{N-1}\up(\vr^k \widehat u^k u^{k})_{i+1/2}~\partial_{i+1/2}\widehat u_h^k.
		\end{split}
	\end{equation}
	Next, we apply the definition of $\up(\cdot)$ and add and subtract to deduce
	\begin{equation*}
		\begin{split}
			&\up(\vr^k \widehat u^k u^{k})_{i+1/2}\partial_{i+1/2}\widehat u_h^k \\
			&\qquad = \left[(\vr^k_i\widehat u^k_i) u^{k,+}_{i+1/2} 
			+ (\vr^k_{i+1}\widehat u^k_{i+1}) u_{i+1/2}^{k,-}\right]~\partial_{i+1/2}\widehat u_h^k \\
			&\qquad = \frac{1}{2}[\vr^k_iu^{k,+}_{i+1/2} + \vr^k_{i+1}u^{k,-}_{i+1/2}]\left(\widehat u_i^k + \widehat u_{i+1}^k\right)\partial_{i+1/2}\widehat u_h^k \\
			&\qquad \qquad \quad  + \frac{1}{2}[\vr^k_iu^{k,+}_{i+1/2} + \vr^k_{i+1}u^{k,-}_{i+1/2}]\left(\widehat u_i^k - \widehat u_{i+1}^k\right)\partial_{i+1/2}\widehat u_h^k \\
			&\qquad  = \up\left(\vr^k u^{k}\right)_{i+1/2}\partial_{i+1/2}~\left(\frac{|\widehat u^k|^2}{2}\right) \\
			&\qquad \qquad\quad - \frac{\Delta x}{2}\left|\up\left(\vr^k u^{k}\right)_{i+1/2}\right|\left|\partial_{i+1/2} \widehat u^k\right|^2
		\end{split}
	\end{equation*}
	We then conclude the proof by setting this identity in \eqref{conv:1}.
\end{proof}

\subsection{Energy estimate}
We are now ready to prove stability of the method.
\begin{proposition}\label{lem:energy}
Let $\{(\vr^k_i, u^k_{i-1/2})\}_{i,k}$,
be the numerical solution obtained through Definition \ref{def:scheme}. 
The following stability estimate holds,
	\begin{align}
			&\max_m\left(\Delta x \sum_i \vr_i^m \frac{|\widehat u_i^m|^2}{2} + \frac{1}{\gamma-1}p(\vr_i^m)\right) \nonumber\\
			&\qquad + \Delta t \Delta x \sum_{k=1}^M\sum_{i}\left|\partial_{i+1/2}u_h^k\right|^2 
			 + \sum_{i=1}^4 \mathcal{N}_i \label{eq:energy}\\
			&\qquad=\Delta x \sum_i \vr_i^0 \frac{|\widehat u_i^0|^2}{2} + \frac{1}{\gamma-1}p(\vr_i^0) \nonumber,
	\end{align}
	where the numerical diffusion terms are
	\begin{align*}
		\mathcal{N}_1 &= (\Delta t)^2\Delta x \sum_{k=1}^M\sum_i p''\left(\vr_i^\ddagger\right)|\partial_t^k \vr_i|^2\\
		\mathcal{N}_2 &= \Delta t(\Delta x)^2\sum_{k=1}^M p''\left(\vr_\dagger^k\right)\left|\partial_{i+1/2}\vr^k\right|^2\left|u^k_{i+1/2}\right| \\
		\mathcal{N}_3 &=  (\Delta t)^2\Delta x\sum_{k=1}^M \sum_i\frac{\vr_{i}^{k-1}}{2}\left|\partial_t^k\widehat u_i \right|^2\\
		\mathcal{N}_4 &= \frac{\Delta t(\Delta x)^2 }{2}\sum_{k=1}^M\sum_{i=0}^{N-1}\left|\up\left(\vr^k u^{k}\right)_{i+1/2}\right|\left|\partial_{i+1/2} \widehat u^k\right|^2.
	\end{align*}
\end{proposition}
\begin{proof}
We begin by multiplying \eqref{num:moment} by $u_{i+1/2}^k\, \Delta x$ 
and sum over all $i$ to obtain
\begin{equation*}
	\begin{split}
		&\Delta x \sum_i \partial_t^k\left(\frac{\vr_i \widehat u_i + \vr_{i+1} \widehat u_{i+1}}{2}\right)u^k_{i+1/2} \\
		&\qquad = - \Delta x \sum_i\left(\frac{\up(\vr^k \widehat u^k u^{k})_{i+3/2} -\up(\vr^k \widehat u^k u^{k})_{i-1/2}}{2\Delta x}\right)u_{i+1/2}^k \\
		&\qquad \qquad -\Delta x \sum_i \left|\partial_{i+1/2}u_h^k\right|^2  - p(\vr_i^k)\partial_{i}u^k_h,
	\end{split}
\end{equation*}
where we have also applied summation by parts together with the boundary condition $u^k_{-1/2} = u^k_{N-1/2} = 0$. 
Next, we apply Lemma \ref{lem:conv} and Lemma \ref{lem:renorm} (with $B(z) = \frac{1}{\gamma-1}p(z)$)
to deduce
\begin{equation}\label{eq:prev}
	\begin{split}
		&\Delta x \sum_i \partial_t^k\left(\frac{\vr_i \widehat u_i + \vr_{i+1} \widehat u_{i+1}}{2}\right)u^k_{i+1/2}
		-\partial_t^k \vr_i\left(\frac{|\widehat u^k|^2}{2}\right) \\
		&\qquad = - \mathcal{N}_2
		 -\Delta x \sum_i \left|\partial_{i+1/2}u_h^k\right|^2  - \frac{1}{\gamma-1}\partial_t^k p(\vr_i)  - \mathcal{P}[p(\vr), u].
	\end{split}
\end{equation}
To proceed, we observe the following identities 
\begin{equation}\label{en:1}
	\Delta x \sum_i \partial_t^k\left(\frac{\vr_i \widehat u_i + \vr_{i+1} \widehat u_{i+1}}{2}\right)u^k_{i+1/2}
	= \Delta x \sum_i \partial_t^k\left(\vr_i \widehat u_i\right)\widehat u^k_i,
\end{equation}
and
\begin{equation}\label{en:2}
	\begin{split}
		&\partial_t^k\left(\vr_i \widehat u_i\right)\widehat u^k_i- \partial_t^k \vr_i\left(\frac{|\widehat u^k|^2}{2}\right)  \\
		&\qquad = \partial_t^k \left(\vr_i \frac{|\widehat u_i|^2}{2}\right) + \frac{\vr_{i}^{k-1}}{2\Delta t}\left|\widehat u_i^k -\widehat u^{k-1}_i \right|^2.
	\end{split}
\end{equation}
By applying \eqref{en:1}-\eqref{en:2} in  \eqref{eq:prev}, we deduce
\begin{equation*}
	\begin{split}
		&\Delta x \sum_i \partial_t^k \left(\vr_i \frac{|\widehat u_i|^2}{2}\right) + \frac{\vr_{i}^{k-1}}{2\Delta t}\left|\widehat u_i^k -\widehat u^{k-1}_i \right|^2
		+ \frac{1}{\gamma-1}\partial_t^k p(\vr_i) \\
		&\qquad + \Delta x \sum_i \left|\partial_{i+1/2}u_h^k\right|^2 + \mathcal{N}_2 + \mathcal{P}[p(\vr), u] =0.
	\end{split}
\end{equation*}
We conclude by multiplying with $\Delta t$, summing over $k=1, \ldots, M$,
and recalling the definition of $\mathcal{P}[\cdot]$.
\end{proof}

The previous stability estimate also provides 
some uniform integrability estimates on various quantities.
To state these, let us introduce the notation
$$
f_h \inb L^p(Q),
$$
to denote the case when $f_h$ is bounded in $L^p$ uniformly 
with respect to both $\Delta t$ and $\Delta x$.

\begin{corollary}\label{cor:energy}
Let $(\vr_h, u_h)$ be the numerical solution constructed through Definition \ref{def:scheme}
and \eqref{def:ext1}-\eqref{def:ext2}. Then, 
\begin{equation*}
	\vr_h \inb L^\infty(0,T;L^\gamma(\Om)), 
	\quad p(\vr_h) \inb L^\infty(0,T;L^1(\Om)),
\end{equation*}
\begin{equation*}
	u_h \in L^2(0,T;W^{1,2}_0(\Om))\cap L^2(0,T;L^\infty(\Om)).
\end{equation*}

As a consequence, 
\begin{equation*}
	\vr_h \widehat u_h \inb L^\infty(0,T;L^\frac{2\gamma}{\gamma+1}(\Om)),
	\quad 
	\vr_h  |\widehat u_h|^2 \inb L^\infty(0,T;L^1(\Om)),
\end{equation*}
\begin{equation*}
	\begin{split}
		\vr_h  u_h \inb L^2(0,T;L^\gamma(\Om), \quad 
		\vr_h  \left| u_h\right|^2 \inb L^2(0,T;L^\frac{2\gamma}{\gamma+1}(\Om)).
	\end{split}
\end{equation*}
\end{corollary}
\begin{proof}
From Proposition \ref{lem:energy}, we have that $\vr_h \inb L^\infty(0,T;L^\gamma(\Om))$,
$u_h \inb L^2(0,T;W^{1,2}_0(\Om))$, and $\vr_h \widehat u_h^2 \inb L^\infty(0,T;L^1(\Om))$.
The remaining bounds follows directly from these using Sobolev embedding
and interpolation estimates.
\end{proof}

\subsection{Control in time}
We end this section by establishing some 
weak control on the time-derivatives. 
For this purpose, let 
\begin{equation*}
	\partial_t^h f_h = \frac{f_h(\cdot) - f_h(\cdot -\Delta t)}{\Delta t}
	= \frac{f^k - f^{k-1}}{\Delta t}, \quad t \in [k \Delta t , (k+1)\Delta t),
\end{equation*}
for all $k = 1, \ldots, M$.

\begin{lemma}\label{lem:timecont}
Let $(\vr_h, u_h)$ be the numerical solution 
constructed through Definition \ref{def:scheme} 
and \eqref{def:ext1}-\eqref{def:ext2}. Then,
\begin{align}
	\partial_t^h\vr_h &\inb L^2\left(0,T;W^{-1, \gamma}(\Om)\right), \label{time:cont}\\
	\partial_t^h(\vr_h \widehat u_h) &\inb L^1\left(0,T;W^{-1, \gamma}(\Om)\right). \label{time:moment}
\end{align}
\end{lemma}
\begin{proof}

1. We first prove \eqref{time:cont}.
Let $\phi \in C^\infty(\Om)$ be arbitrary 
and define 
\begin{equation*}
	\phi_i = \frac{1}{\Delta x}\int_{x_{i-1/2}}^{x_{i+1/2}}
	\phi(y)~dy, \quad i=0, \ldots, N-1.
\end{equation*}
Now, multiply the continuity scheme \eqref{num:cont} with $\Delta x \phi_i$, 
sum over all $i$, and apply summation by parts, to deduce
\begin{equation}\label{time:eq}
	\begin{split}
		\Delta x\sum_i 
		\left(\partial_t^k {\vr_i^k}\right)\phi_i 
		&=   \Delta x\sum_i 
		\up(\vr^k u^k)_{i+1/2}\left(\frac{\phi_{i+1} - \phi_i}{\Delta x}\right) \\
		&\leq C
		\|\vr_h\|_{L^\gamma(\Om)}
		\|u_h\|_{L^\infty(\Om)}
		\|\phi_x\|_{L^\frac{\gamma}{\gamma-1}(\Om)}.
	\end{split}
\end{equation}
The last inequality is an application of the H\"older inequality.
Now, since \eqref{time:eq} holds for all $\phi$, we can conclude 
that it continues to hold for all $\phi \in W^{1,\frac{\gamma}{\gamma-1}(\Om)}$
By multiplying with $\Delta t$ and summing over all k,
we obtain 
\begin{equation*}
	\begin{split}
		&\left\|\partial_t^k\vr_h\right\|_{L^2(0,T;W^{-1, \gamma}(\Om))} 
		 \leq \|\vr_h\|_{L^\infty(0,T;L^\gamma(\Om))}
		\|u_h\|_{L^2(0,T;L^\infty(\Om))},
	\end{split}
\end{equation*} 
which is \eqref{time:cont}.

2. Let $\phi \in W^{1,\infty}_0(\Om)$ be arbitrary and define
$$
v_{i-1/2} = v(x_{i-1/2}), \quad i=0, \ldots, N.
$$
By multiplying the momentum scheme \eqref{num:moment}
with $\Delta x v_{i+1/2}$, summing over $i$, and 
applying summation by parts, we obtain
\begin{equation*}
	\begin{split}
		&\Delta x\sum_i \partial_t^k(\vr_i\widehat u_i
		+\vr_{i+1}\widehat u_{i+1})v_{i+1/2} \\
		&\quad = \Delta x\sum_i \up(\vr^k \widehat u^k u^k)
		\left(\frac{v_{i+3/2} -v_{i-1/2}}{2\Delta x}\right) \\
		&\qquad +\Delta x\sum_i \mu\partial_i u^k
		\left(\frac{v_{i+1/2}-v_{i-1/2}}{\Delta x}\right)
		- p(\vr_i^k)\left(\frac{v_{i+1/2}-v_{i-1/2}}{\Delta x}\right) \\
		&\quad\leq C\|\vr^k_h\|_{L^\gamma(\Om)}
		\|u^k_h\|_{L^\infty(\Om)}^2\left\|v_x\right\|_{L^\frac{\gamma}{\gamma-1}(\Om)} \\
		&\qquad \quad + \mu \|(u^k_h)_x\|_{L^2(\Om)}
		\left\|v_x\right\|_{L^2(\Om)} + \|p(\vr^k_h)\|_{L^\infty(\Om)}
		\left\|v_x\right\|_{L^\infty(\Om)}.
	\end{split}
\end{equation*}
Next, we multiply with $\Delta t$ and sum over all $k$ to conclude
\begin{equation*}
	\begin{split}
		&\Delta t\sum_k \left|\Delta x\sum_i \partial_t^k(\vr_i\widehat u_i
		+\vr_{i+1}\widehat u_{i+1})v_{i+1/2}\right| \\
		& \leq C\|v_x\|_{L^\infty(\Om)}
		\left(\|u_h\|_{L^2(0,T;L^\infty(\Om))}^2\|\vr_h\|_{L^\infty(0,T;L^\gamma(\Om))}\right. \\
		&\qquad \left. \|(u_h)_x\|_{L^2(0,T;L^2(\Om))}
		+ \|p(\vr_h)\|_{L^\infty(0,T;L^1(\Om))} \right) 
		 \leq C\|v_x\|_{L^\infty(\Om)},
	\end{split}
\end{equation*}
where the last inequality is Corollary \ref{cor:energy}.
Finally, we calculate
\begin{align}
		&\int_0^T\left|\int_0^L \partial_t^k(\vr_h \widehat u_h)v~dx\right|dt \nonumber\\
		&\qquad \leq
		\Delta t\Delta x\sum_k\left|\sum_i \partial_t^k(\vr_i\widehat u_i)
		\left(\frac{v_{i+1/2}+v_{i-1/2}}{2}\right)\right| \nonumber\\
		&\qquad \quad
		+\Delta t\sum_k\left|\sum_i \partial_t^k(\vr_i\widehat u_i)
		\int_{x_{i-1/2}}^{x_{i+1/2}}\left(\frac{v_{i+1/2}+v_{i-1/2}}{2}- v(x)\right)~ dx\right|\nonumber \\
		&\qquad =
		\Delta t\Delta x\sum_k\sum_i \partial_t^k(\vr_i\widehat u_i
		+\vr_{i+1}\widehat u_{i+1})v_{i+1/2} \nonumber\\
		&\qquad \quad
		+\Delta t\sum_k\left|\sum_i \partial_t^k(\vr_i\widehat u_i)
		\int_{x_{i-1/2}}^{x_{i+1/2}}\left(\frac{v_{i+1/2}+v_{i-1/2}}{2}- v(x)\right)~ dx\right|\nonumber \\
		&\qquad \leq C\left(1 + \|\vr_h\widehat u_h\|_{L^1(0,T;L^1(\Om))}
		\right)\|v_x\|_{L^\infty(\Om)}, \nonumber
\end{align}
form which \eqref{time:moment} follows.

\end{proof}

\section{The effective viscous flux}
The purpose of this section is to derive an 
equation for the quantity
$$
	\mathcal{F}_h = \mu(u_h)_x - p(\vr_h).
$$
This quantity is often termed the \emph{effective viscous flux}
and stands at the center of both the available existence results (in more than 1D) \cite{Lions:1998ga, Feireisl}
and the available numerical results (cf.~\cite{Karper}). 
Specifically, both higher (than $L^1$) integrability 
on the pressure and strong convergence of the density 
is proved using the upcoming equation (see Proposition \ref{pro:eff} for details).

To derive the desired equation, we shall need the following discrete Neumann Laplace operator 
\begin{equation*}
	\begin{split}
			-\Delta_{i}q_h &= -\partial_{i}\partial_{i+1/2} q_h = f_i, \\
			q_{N} &= q_{N-1}, \\
			q_{-1} &= q_1.
	\end{split}
\end{equation*}
and we observe that $-\Delta_i$ is nothing but the standard 3-point 
Laplacian on the density grid. Moreover, the Neumann condition 
is realized by adding a shadow cell on both sides of the domain $(0,L)$.
Due to the Neumann condition, $\Delta_i$ is only well-defined
for sources $f_h$ of zero mean.

In the upcoming analysis, we shall not need $\Delta_i$ directly, 
but  its discrete derivative:
\begin{equation*}
	\partial_{i+1/2}\Delta_i^{-1}\left[f_h^k\right] = -\partial_{i+1/2}q_h^k,
\end{equation*}
where we observe that the Neumann condition renders $\partial_{N-1/2}\Delta_i^{-1}\left[f_h^k\right] = 0$
and $\partial_{-1/2}\Delta_i^{-1}\left[f_h^k\right] = 0$.
We shall also need the inverse of the $\Delta_{i+1/2}$ operator
occurring in the momentum scheme \eqref{num:moment}:
\begin{equation*}
	-\Delta^{-1}_{i+1/2}[v_h] = w_{i+1/2}, 
\end{equation*}
where $w_h$ solves the linear Dirichlet  system
\begin{equation*}
	\begin{split}
		-\Delta_{i+1/2}w_h =-\partial_{i+1/2}\partial_i w_h  = v_{i+1/2}, \qquad w_{-1/2} = w_{N-1/2} = 0.
	\end{split}
\end{equation*}
We will mostly be interested in the discrete derivative
\begin{equation*}
	\partial_i \Delta_{i+1/2}^{-1}[v_h] = - \partial_i w_h,
\end{equation*}

The following result follows from standard summation by parts.
\begin{lemma}\label{lem:Aop}
The following duality holds,
	\begin{equation*}
		\begin{split}
				&\Delta x\sum_i v_{i+1/2}	\partial_{i+1/2}\Delta_i^{-1}\left[f_h\right] =
				 -\Delta x \sum_i \partial_i \Delta^{-1}_{i+1/2}[v_h] f_i.
		\end{split}
	\end{equation*}
\end{lemma}
\begin{proof}
By direct calculation, using the definition of $\Delta_i$ and $\Delta_{i+1/2}$, 
and the respective boundary conditions, we deduce
\begin{equation*}
	\begin{split}
		&\Delta x\sum_i v_{i+1/2}	\partial_{i+1/2}\Delta_i^{-1}\left[f_h\right]  \\
		&\qquad =\Delta x\sum_i \left(\partial_{i+1/2}\partial_i \Delta^{-1}_{i+1/2}[v_{h}]\right)\left(\partial_{i+1/2}\Delta_i^{-1}\left[f_h\right]\right) \\
		&\qquad=-\Delta x\sum_i \left(\partial_i\Delta^{-1}_{i+1/2}[v_{h}]\right)\left(\partial_i \partial_{i+1/2}\Delta^{-1}\left[f_h\right]\right) \\
		&\qquad=-\Delta x \sum_i \left(\partial_i\Delta^{-1}_{i+1/2}[v_{h}]\right)f_i,
	\end{split}
\end{equation*}
which concludes the proof.
\end{proof}

The following proposition gives
the effective viscous flux equation we shall need in the convergence analysis. 
The two error terms appearing in \eqref{eq:eff} will 
be bounded below.

\begin{proposition}\label{pro:eff}
Let $(\vr_h, u_h)$ be the numerical solution 
constructed through Definition  \ref{def:scheme} and \eqref{def:ext1}-\eqref{def:ext2}.
For any $m=1, \ldots, M$, we have that 
\begin{equation}\label{eq:eff}
	\begin{split}
		&-\Delta t \Delta x \sum_{k=0}^m\sum_i \left(\mu \partial_i u^k_h - p\left(\vr_i^k\right)\right)\left(\vr_i^k - \frac{\mathcal{M}}{L}\right) \\
		&\qquad =  \Delta t\Delta x \sum_{k=0}^m\sum_i \up(\vr^k \widehat u^k u^k)_{i+1/2}\left(\frac{\mathcal{M}}{L}\right)\\
		&\qquad \quad- \Delta x \sum_i\partial_{i}\Delta^{-1}_{i+1/2}\left[\frac{\vr^{m}_i \widehat u^{m}_i + \vr^{m}_{i+1} \widehat u^{m}_{i+1}}{2}\right]
		\vr^M_i\\
		&\qquad \quad + \Delta x \sum_i\partial_{i}\Delta^{-1}_{i+1/2}\left[\frac{\vr^{0}_i \widehat u^{0}_i + \vr^{0}_{i+1} \widehat u^{0}_{i+1}}{2}\right]
		\vr^1_i + E_1^m + E_2^m,
	\end{split}
\end{equation}	
where the numerical error terms $E_1^m$ and $E_2^m$  are given by \eqref{eff:E1} and \eqref{eff:E2}, respectively.
\end{proposition}
\begin{proof}
Let $m = 1, \ldots, M$ be arbitrary.
By multiplying the momentum scheme \eqref{num:moment} with $v_{i+1/2}(\Delta t \Delta x)$, where
 $$
	v_{i+1/2}^k = \partial_{i+1/2}\Delta^{-1}_i\left[\vr^k_i - \frac{\mathcal{M}}{L}\right], \qquad \mathcal{M} = \int_0^L \vr^0~dx,
$$
and summing over all $i$ and $k=0, \ldots, m$, we obtain the starting point
\begin{equation}\label{eff:start}
	\begin{split}
		&\Delta t\Delta x\sum_{k=0}^m\sum_i \left(\mu \Delta_{i+1/2}u^{k}_h - \, 
		\partial_{i+1/2}p(\vr^{k}_h)\right)v_{i+1/2}^k \\
		&\qquad =\Delta t\Delta x\sum_{k=0}^m\sum_i\partial_t^k\left(\frac{\vr_i \widehat u_i + \vr_{i+1} \widehat u_{i+1}}{2}\right)v_{i+1/2}^k\\
		& \qquad
		\quad +\Delta t\Delta x\sum_{k=0}^m\sum_i\left(\frac{\up(\vr^k \widehat u^k u^{k})_{i+3/2} -\up(\vr^k \widehat u^k u^{k})_{i-1/2}}{2\Delta x}\right)
		v_{i+1/2}^k \\
		&\qquad= S_1 + S_2.
	\end{split}
\end{equation}
Here, we have introduced the quantities
\begin{equation*}
	\begin{split}
		S_1 &= \Delta t\Delta x\sum_{k=0}^m\sum_i\partial_t^k\left(\frac{\vr_i \widehat u_i + \vr_{i+1} \widehat u_{i+1}}{2}\right)v_{i+1/2}^k \\
		S_2 &= \Delta t\Delta x\sum_{k=0}^m\sum_i\left(\frac{\up(\vr^k \widehat u^k u^{k})_{i+3/2} -\up(\vr^k \widehat u^k u^{k})_{i-1/2}}{2\Delta x}\right)v_{i+1/2}^k
	\end{split}
\end{equation*}

We will now rewrite the sums $S_1$ and $S_2$ using the definition of $v_{i+1/2}^k$
and the continuity scheme. We will treat  the least involved term, namely $S_2$, first.
For this purpose, we apply summation by parts and the definition of $v_{i+1/2}^k$ 
to calculate
\begin{equation}\label{eff:S2F}
	\begin{split}
		S_2 &= - \Delta t\Delta x \sum_{k=0}^m\sum_i \up(\vr^k \widehat u^k u^k)_{i+1/2}\left(\frac{v_{i+3/2} - v_{i-1/2}}{2\Delta x}\right) \\
			&= -\Delta t\Delta x \sum_{k=0}^m \sum_i \up(\vr^k \widehat u^k u^k)_{i+1/2}\left(\frac{\vr_{i+1}^k + \vr_i^k}{2}\right) \\
			&\qquad +\Delta t\Delta x \sum_{k=0}^m\sum_i \up(\vr^k \widehat u^k u^k)_{i+1/2}\left(\frac{\mathcal{M}}{L}\right).
	\end{split}
\end{equation}

For the $S_1$ term, we first apply summation by parts in time 
followed by an application of Lemma \ref{lem:Aop} to deduce
\begin{equation}\label{eff:S1}
	\begin{split}
		S_1 &= \Delta t\Delta x\sum_{k=0}^m\sum_i\partial_t^k\left(\frac{\vr_i \widehat u_i + \vr_{i+1} \widehat u_{i+1}}{2}\right)
		 \left(\partial_{i+1/2}\Delta^{-1}_i\left[\vr^k_i - \frac{\mathcal{M}}{L}\right]\right)\\
		&= -\Delta t \Delta x \sum_{k=0}^m \sum_i\left(\frac{\vr^{k-1}_i \widehat u^{k-1}_i + \vr^{k-1}_{i+1} \widehat u^{k-1}_{i+1}}{2}\right)
		\left(\partial_{i+1/2}\Delta^{-1}_i\left[\partial_t^k \vr_i \right]\right) \\
		&= \Delta t \Delta x \sum_{k=0}^m \sum_{i} \partial_{i}\Delta^{-1}_{i+1/2}\left[\frac{\vr^{k-1}_i \widehat u^{k-1}_i + \vr^{k-1}_{i+1} \widehat u^{k-1}_{i+1}}{2}\right]
		\partial_t^k\vr_i \\
		&\qquad\qquad - \Delta x \sum_i\partial_{i}\Delta^{-1}_{i+1/2}\left[\frac{\vr^{m}_i \widehat u^{m}_i + \vr^{m}_{i+1} \widehat u^{m}_{i+1}}{2}\right]
		\vr^M_i\\
		&\qquad\qquad + \Delta x \sum_i\partial_{i}\Delta^{-1}_{i+1/2}\left[\frac{\vr^{0}_i \widehat u^{0}_i + \vr^{0}_{i+1} \widehat u^{0}_{i+1}}{2}\right]
		\vr^1_i
	\end{split}
\end{equation}

Next, we multiply the continuity scheme \eqref{num:cont} by $\Delta t \Delta x q_i^k$, where
$$
 q_i^k = \partial_{i}\Delta^{-1}_{i+1/2}\left[\frac{\vr^{k-1}_i \widehat u^{k-1}_i 
												  + \vr^{k-1}_{i+1} \widehat u^{k-1}_{i+1}}{2}\right],
$$
and sum over all $i$ and $k=0, \ldots, m$ to obtain
\begin{equation}\label{eff:time}
	\begin{split}
		&\Delta t \Delta x \sum_{k=0}^m\sum_{i} \partial_{i}\Delta^{-1}_{i+1/2}\left[\frac{\vr^{k-1}_i \widehat u^{k-1}_i 
										+ \vr^{k-1}_{i+1} \widehat u^{k-1}_{i+1}}{2}\right]\partial_t^k\vr_i \\
		&\quad = -\Delta t \Delta x\sum_{k=0}^m\sum_{i} \partial_{i}\up(\vr^k u^k)\left(\partial_{i}\Delta^{-1}_{i+1/2}
							\left[\frac{\vr^{k-1}_i \widehat u^{k-1}_i + \vr^{k-1}_{i+1} \widehat u^{k-1}_{i+1}}{2}\right]\right) \\
		&\quad = \Delta t \Delta x\sum_{k=0}^m\sum_{i} \up(\vr^k u^k)_{i+1/2}\left(\frac{\vr^{k-1}_i \widehat u^{k-1}_i + \vr^{k-1}_{i+1} \widehat u^{k-1}_{i+1}}{2}\right) \\
		&\quad = \Delta t \Delta x \sum_{k=0}^m\sum_{i} \up(\vr^k u^k)_{i+1/2}\left(\frac{\vr^{k}_i \widehat u^{k}_i + \vr^{k}_{i+1} \widehat u^{k}_{i+1}}{2}\right) + E_1^m,
	\end{split}
\end{equation}
where the last identity follows from summation by parts, the definition of $\Delta^{-1}_{i+1/2}$, 
and where we have introduced the error term 
\begin{equation}\label{eff:E1}
	\begin{split}
		E_1^m&= 
		 -\Delta t \Delta x \sum_{k=0}^m\sum_{i} \up(\vr^k u^k)_{i+1/2}
			\left(\frac{\vr^{k}_i \widehat u^{k}_i + \vr^{k}_{i+1} \widehat u^{k}_{i+1}}{2}\right. \\
			&\qquad \qquad\qquad \qquad\qquad \qquad\qquad \qquad-\left. \frac{\vr^{k-1}_i \widehat u^{k-1}_i+ \vr^{k-1}_{i+1} \widehat u^{k-1}_{i+1}}{2}\right)		
	\end{split}
\end{equation}

To conclude the proof, we shall need to further rewrite \eqref{eff:time}.
By adding and subtracting, we obtain the identity
\begin{equation*}
	\begin{split}
		&\up(\vr^k u^k)_{i+1/2}\left(\frac{\vr^{k}_i \widehat u^{k}_i + \vr^{k}_{i+1} \widehat u^{k}_{i+1}}{2}\right) \\
		&= \left(\vr_i^k  u_{i+1/2}^{+,k} + \vr_{i+1}^k  u_{i+1/2}^{-,k}\right)\left(\frac{\vr^{k}_i \widehat u^{k}_i + \vr^{k}_{i+1} \widehat u^{k}_{i+1}}{2}\right) \\
		&= \vr^k_i u_{i+1/2}^{+, k}\widehat u_i^k \left(\frac{\vr^k_i + \vr^k_{i+1}}{2}\right)
		+\vr^k_i u_{i+1/2}^{+,k}\vr_{i+1}^k\left(\frac{\widehat u_{i+1}^k -\widehat u_{i}^k }{2}\right) \\
		&\quad +\vr^k_{i+1} u_{i+1/2}^{-, k}\widehat u_{i+1}^k\left(\frac{\vr^k_i + \vr^k_{i+1}}{2}\right)
		-\vr^k_{i+1} u_{i+1/2}^{-, k}\vr_i^k\left(\frac{\widehat u_{i+1}^k -\widehat u_{i}^k }{2}\right) \\
		&= \up(\vr^k \widehat u^k u^k)\left(\frac{\vr^k_i + \vr^k_{i+1}}{2}\right)
		+ \vr_i^k\vr_{i+1}^k \left|u_{i+1/2}^k\right|\left(\frac{\widehat u_{i+1}^k -\widehat u_{i}^k }{2}\right).
	\end{split}
\end{equation*}
Consequently, by combining this identity \eqref{eff:time} and \eqref{eff:S1}, we conclude
\begin{equation}\label{eff:S1F}
	\begin{split}
		S_1 &= \Delta t \Delta x \sum_{k=0}^m\sum_{i}\up(\vr^k \widehat u^k u^k)\left(\frac{\vr^k_i + \vr^k_{i+1}}{2}\right)\\
		&\qquad\qquad - \Delta x \sum_i\partial_{i}\Delta^{-1}_{i+1/2}\left[\frac{\vr^{m}_i \widehat u^{m}_i + \vr^{m}_{i+1} \widehat u^{m}_{i+1}}{2}\right]
		\vr^M_i\\
		&\qquad\qquad + \Delta x \sum_i\partial_{i}\Delta^{-1}_{i+1/2}\left[\frac{\vr^{0}_i \widehat u^{0}_i + \vr^{0}_{i+1} \widehat u^{0}_{i+1}}{2}\right]
		\vr^1_i + E_1^m + E_2^m,
	\end{split}
\end{equation}
where 
\begin{equation}\label{eff:E2}
	E_2^m = \Delta t \Delta x \sum_{k=0}^m\sum_{i}\vr_i^k\vr_{i+1}^k \left|u_{i+1/2}^k\right|\left(\frac{\widehat u_{i+1}^k -\widehat u_{i}^k }{2}\right).
\end{equation}

Finally, we apply \eqref{eff:S2F} and \eqref{eff:S1F} to \eqref{eff:start} to discover
\begin{equation*}
	\begin{split}
		&\Delta t\Delta x\sum_{k=0}^m\sum_i \left(\mu \Delta_{i+1/2}u^{k}_h - \, 
		\partial_{i+1/2}p(\vr^{k}_h)\right)v_{i+1/2}^k \\
		&=  \Delta t\Delta x \sum_{k=0}^m\sum_i \up(\vr^k \widehat u^k u^k)_{i+1/2}\left(\frac{\mathcal{M}}{L}\right)\\
		&\quad- \Delta x \sum_i\partial_{i}\Delta^{-1}_{i+1/2}\left[\frac{\vr^{m}_i \widehat u^{m}_i + \vr^{m}_{i+1} \widehat u^{m}_{i+1}}{2}\right]
		\vr^m_i\\
		&\qquad\qquad + \Delta x \sum_i\partial_{i}\Delta^{-1}_{i+1/2}\left[\frac{\vr^{0}_i \widehat u^{0}_i + \vr^{0}_{i+1} \widehat u^{0}_{i+1}}{2}\right]
		\vr^1_i + E_1^m + E_2^m.
	\end{split}
\end{equation*}
This and a final summation by parts, 
\begin{equation*}
	\begin{split}
		&\Delta t\Delta x\sum_{k=0}^m\sum_i \left(\mu \Delta_{i+1/2}u^{k}_h - \, 
		\partial_{i+1/2}p(\vr^{k}_h)\right)v_{i+1/2}^k \\
		&=-\Delta t \Delta x \sum_{k=0}^m\sum_{i} \left(\mu \partial_i u^k_h - p\left(\vr_i^k\right)\right)\left(\vr_i^k - \frac{\mathcal{M}}{L}\right),
	\end{split}
\end{equation*}
concludes the proof.
\end{proof}

\subsection{Bound on the error terms in (\ref{eq:eff})}
In order for the previous proposition to be useful, 
we will need to prove suitable bounds on the 
error terms $E^h_1$ and $E^h_2$. It is at this stage
we will need to impose some requirements 
on the value of $\gamma$. No other results 
in the paper imposes any unphysical restrictions on $\gamma$.

We begin by deriving a bound on $E_2^h$.

\begin{lemma}\label{lem:E2}
Assume that the adiabatic exponent satisfies
\begin{equation*}
	 \gamma > \frac{4}{3},
\end{equation*}
and let $E_2^h$ be given by \eqref{eff:E2}. There is a constant $C> 0$, 
independent of $\Delta t$ and $\Delta x$, such that for any $m=1, \ldots, M$,
\begin{equation*}
	|E^m_2| \leq (\Delta x)^\frac{3\gamma-4}{2\gamma},
\end{equation*}
\end{lemma}

\begin{proof}
By two applications of the Cauchy-Schwartz inequality and a 
standard inverse estimate, we deduce
	\begin{equation*}\label{E2:start}
		\begin{split}
			|E_2^m|\leq |E_2^M| &\leq \Delta t \Delta x \sum_k\sum_{i}\vr_i^k\vr_{i+1}^k \left|u_{i+1/2}^k\right|\left|\frac{\widehat u_{i+1}^k -\widehat u_{i}^k }{2}\right| \\
			&\leq (\Delta x)C\left(\Delta t \Delta x \sum_k\sum_{i}\left|\partial_{i}u_{i+1/2}^k\right|^2\right)^\frac{1}{2} \\
			&\qquad \qquad \times \|u_h\|_{L^2(0,T;L^\infty(\Om))}\|\vr_h\|_{L^\infty(0,T;L^4(\Om))}^2 \\
			&\leq (\Delta x)^1\left(\Delta x\right)^{2\left(\frac{1}{4} - \frac{1}{\gamma}\right)}C \leq (\Delta x)^{\frac{3\gamma-4}{2\gamma}}C,
		\end{split}
	\end{equation*}
	where the norms are bounded due to Corollary \ref{cor:energy}.
\end{proof}

We are now ready to bound the other error term in \eqref{eq:eff}.
The following lemma is the sole reason for the requirement 
$\gamma > \frac{3}{2}$.

\begin{lemma}\label{lem:E1}
Assume that the adiabatic coefficient satisfies 
$$
	\gamma > \frac{3}{2},
$$
and let $E_1^h$  be given by \eqref{eff:E1}. 
There exists a constant $C>0$, independent of discretization 
parameters, such that for any $m=0, \ldots, M$,
\begin{equation*}
	\begin{split}
		|E^m_1| &\leq h^\frac{2\gamma - 3}{2\gamma}C.
	\end{split} 
\end{equation*}	
\end{lemma}
\begin{proof}
From \eqref{eff:E1}, we have that
\begin{equation}\label{E1:rstart}
	\begin{split}
		\left|E_1^m\right| &=
		\left|(\Delta t)^2 \Delta x \sum_{k=0}^m\sum_{i} \up(\vr^k u^k)_{i+1/2}
			\left(\frac{ \partial_t^k(\vr_i \widehat u_i) + \partial_t^{k}(\vr_{i+1} \widehat u_{i+1})}{2}\right)\right| \\
	\end{split}
\end{equation}
Let us now examine one of the terms in $E_1^h$. By adding and subtracting, 
we write
\begin{equation}\label{E1:start}
	\begin{split}
		&\left|(\Delta t)^2 \Delta x\sum_{k=0}^m\sum_{i}\up(\vr^k u^k)_{i+1/2} \partial_t^k(\vr_i \widehat u_i)\right| \\
		&\quad\leq (\Delta t)^2 \Delta x\sum_{k=0}^m\sum_{i} \left|\up(\vr^k u^k)_{i+1/2} \widehat u_i^k\right| |\partial_t^k \vr_i| \\
		&\qquad + (\Delta t)^2 \Delta x\sum_{k=0}^m\sum_{i}\left|\up(\vr^k u^k)_{i+1/2}\right|\vr_i^{k-1}\left|\partial_t^k \widehat u_i\right| 
		 =: S_1 + S_2.
	\end{split}
\end{equation}
To bound the $S_1$ term, 
we multiply and divide by $\sqrt{p}''(\vr^\ddagger_i)$ and apply the Cauchy-Schwartz 
inequality 
\begin{equation}\label{eq:E11}
	\begin{split}
		S_1 &= (\Delta t)^2 \Delta x\sum_{k=0}^m\sum_{i} \left|\up(\vr^k u^k)_{i+1/2} \widehat u_i^k\right| |\partial_t^k \vr_i| \\
		& \leq 
		(\Delta t)^\frac{1}{2} C\left(\Delta x \Delta t\sum_{k=0}^m\sum_{i} \left|\up(\vr^k u^k)_{i+1/2} \widehat u_i^k\right|^2 (\vr^\ddagger_i)^{2-\gamma} \right)^\frac{1}{2} \\
		&\qquad\qquad \times \left((\Delta t)^2 \Delta x\sum_{k=0}^m\sum_{i} p''(\vr_i^\ddagger)\left|\partial_t^k \vr_i\right|^2\right)^\frac{1}{2}.
	\end{split}
\end{equation}
Next, we make several applications of the H\"older inequality together
with standard inverse estimates, to deduce
\begin{equation}\label{eq:E12}
	\begin{split}
		&\Delta x \Delta t\sum_{k=0}^m\sum_{i} \left|\up(\vr^k u^k)_{i+1/2} \widehat u_i^k\right|^2 (\vr^\ddagger_i)^{2-\gamma} \\
		&\qquad\leq \|u_h\|_{L^\infty(0,T;L^2(\Om))}^2\|\vr_h\|_{L^\infty(0,T;L^\infty(\Om))}^{2-\gamma}\|\vr_h u_h\|_{L^\infty(0,T;L^2(\Om))}^2 \\
		&\qquad \leq C (\Delta x)^{-\frac{2-\gamma}{\gamma}}\|\vr_h\|_{L^\infty(0,T;L^\gamma(\Om))}^{2-\gamma}\\
		&\qquad  \qquad \times \, (\Delta x)^{2\left(\frac{1}{2} - \frac{\gamma+1}{2\gamma}\right)}\|\vr_h u_h\|_{L^\infty(0,T;L^\frac{2\gamma}{\gamma+1}(\Om))}^2 \\
		&\qquad \leq (\Delta x)^{-1+ \frac{2\gamma-3}{\gamma}}C,
	\end{split}
\end{equation}
where we have used Corollary \ref{cor:energy} to conclude the last inequality.

Combining \eqref{eq:E12}-\eqref{eq:E11} and 
recalling that $\Delta t = \Delta x$, yields
\begin{equation}\label{eq:ES1}
	S_1 \leq (\Delta x)^{\frac{2\gamma-3}{2\gamma}}.
\end{equation}

Next, we turn to the $S_2$ term in \eqref{E1:start}. An application 
of the Cauchy-Schwartz inequality yields
\begin{equation}\label{eq:S21}
	\begin{split}
		S_2 & =(\Delta t)^2 \Delta x\sum_{k=0}^m\sum_{i}\left|\up(\vr^k u^k)_{i+1/2}\right|\vr_i^{k-1}\left|\partial_t^k \widehat u_i\right| \\
		&\leq \left((\Delta t)^2 \Delta x \sum_{k=0}^m\sum_{i}\vr_{i-1}^k \left|\partial_t^k \widehat u_i\right|^2\right)^\frac{1}{2} \\
		&\qquad \times (\Delta t)^\frac{1}{2}\left(\Delta t \Delta x \sum_{k=0}^m\sum_{i}\left|\up(\vr^k u^k)_{i+1/2}\right|^2\vr_i^{k-1}\right)^\frac{1}{2} \\
		&\leq (\Delta t)^\frac{1}{2}C \|\vr_h u_h\|_{L^2(0,T;L^2(\Om))}\|\vr_h\|_{L^\infty(0,T;L^\infty(\Om))}^\frac{1}{2}
	\end{split}
\end{equation}
Next, we proceed as in \eqref{eq:E12} to discover
\begin{equation}\label{eq:S22}
	\begin{split}
		&\Delta t \Delta x \sum_{k=0}^m\sum_{i}\left|\up(\vr^k u^k)_{i+1/2}\right|^2\vr_i^{k-1} \\
		&\qquad \leq \|\vr_h u_h\|_{L^2(0,T;L^2(\Om))}^2\|\vr_h\|_{L^\infty(0,T;L^\infty(\Om))} \\
		&\qquad \leq C(\Delta x)^{2\left(\frac{1}{2} - \frac{1}{\gamma}\right)}\|\vr_h u_h\|_{L^2(0,T;L^\gamma(\Om))}^2
		(\Delta x)^{-\frac{1}{\gamma}}\|\vr_h\|_{L^\infty(0,T^;L^\gamma(\Om))} \\
		&\qquad \leq (\Delta x)^{-1 + \frac{2\gamma - 3}{\gamma}}C,
	\end{split}
\end{equation}
where we have used Corollary \ref{cor:energy} in the last inequality.

By combining \eqref{eq:S21} and\eqref{eq:S22}, and recalling that $\Delta t = \Delta x$, 
we conclude 
\begin{equation}\label{eq:ES2}
	S_2 \leq (\Delta x)^{\frac{2\gamma-3}{2\gamma}}C.
\end{equation}
By setting \eqref{eq:ES1} and \eqref{eq:ES2} in \eqref{E1:start} we obtain
\begin{equation*}
	S_1 + S_2 \leq (\Delta x)^\frac{2\gamma-3}{2\gamma}C,
\end{equation*}
and hence we have the desired bound for the first term in \eqref{E1:rstart}.
The second term in \eqref{E1:rstart} can be bounded by the exact same arguments.
\end{proof}

\subsection{Higher integrability on the density}
From Corollary \ref{cor:energy}, we only know that 
$p(\vr_h)$ is uniformly (in $\Delta t$ and $\Delta x$)
bounded in $L^\infty(0,T;L^1(0,L))$. Hence, it is unclear 
whether $p(\vr_h)$ actually converges to an integrable function. 
In the following lemma, we prove that the pressure has
more integrability than provided by the energy estimate.

\begin{lemma}\label{lem:}
Let $(\vr_h, u_h)$ be the numerical solution constructed
using Definition \ref{def:scheme} 
and \eqref{def:ext1}-\eqref{def:ext2}, 
with 
$$
	\gamma > \frac{3}{2}.
$$
There is a constant $C>0$, independent of discretization parameters, such that
\begin{equation*}
	a\int_0^T\int_0^L \vr_h^{\gamma+1}~ dxdt \leq C.
\end{equation*}
In other words,  $p(\vr_h) \in L^\infty(0,T;L^\frac{\gamma+1}{\gamma}(\Om))$.
\end{lemma}
\begin{proof}
First, we rewrite the equation \eqref{eq:eff} in the form
\eqref{eq:eff}
\begin{equation*}
	\begin{split}
		&\int_0^T\int_0^L p(\vr_h)\vr_h~dxdt \\
		&\qquad 
		= \int_0^T\int_0^L p(\vr_h)\frac{\mathcal{M}}{L}~dx 
		 + \mu\int_0^T\int_0^L (u_h)_x \vr_h~dxdt \\
		&\qquad\quad +\Delta t\Delta x \sum_{k=0}^m\sum_i \up(\vr^k \widehat u^k u^k)_{i+1/2}\left(\frac{\mathcal{M}}{L}\right)\\
		&\qquad\quad  - \Delta x \sum_i\partial_{i}\Delta^{-1}_{i+1/2}\left[\frac{\vr^{M}_i \widehat u^{M}_i + \vr^{M}_{i+1} \widehat u^{M}_{i+1}}{2}\right]
		\vr^M_i\\
		&\qquad \quad + \Delta x \sum_i\partial_{i}\Delta^{-1}_{i+1/2}\left[\frac{\vr^{0}_i \widehat u^{0}_i + \vr^{0}_{i+1} \widehat u^{0}_{i+1}}{2}\right]
		\vr^1_i 
		 + E_1^M + E_2^M.
	\end{split}
\end{equation*}
To bound the terms involving the discrete inverse Laplacian, 
we shall use the elementary bound
\begin{equation*}
	\left\|\partial_i\Delta^{-1}_{i+1/2}
	\left[\vr^k_h \widehat u_h^k\right]\right\|_{L^\infty(0,T;L^\infty(\Om))}
	\leq C\|\vr_h u_h\|_{L^\infty(0,T;L^1(\Om))}.
\end{equation*}
Together with the H\"older inequality and Lemmas \ref{lem:E1} and 
\ref{lem:E2}, this readily provides the bound
\begin{equation*}
	\begin{split}
	&\int_0^T\int_0^L p(\vr_h)\vr_h~dxdt \\
	&\qquad \leq \frac{\mathcal{M}}{L}
	\|\vr_h\|_{L^\infty(0,T;L^\gamma(\Om))}^\gamma
	+ \frac{\mu}{\eps}\left\|(u_h)_x\right\|_{L^2(0,T;L^2(\Om))}^2 \\
	&\qquad \quad+ \eps \left\|\vr_h\right\|_{L^{\gamma+1}(0,T;L^{\gamma+1}(\Om))}^{\gamma+1}
	 + \frac{\mathcal{M}C}{L}\left\|\vr_h u_h^2\right\|_{L^1(0,T;L^1(\Om))}\\
	&\qquad \quad +\mathcal{M}C\|\vr_h \widehat u_h\|_{L^\infty(0,T;L^1(\Om))} 
	 + (\Delta x)^\frac{2\gamma-3}{2\gamma}C.
	\end{split}
\end{equation*}	
The proof is completed by fixing $\epsilon$ sufficiently small.

\end{proof}

\section{Weak convergence}
In this section, we will pass to the limit
in the numerical method and prove 
that the limit is almost a weak solution 
to the compressible Navier-Stokes equations. 
Our starting point is that Corollary \ref{cor:energy}
allow us to assert the existence of functions
\begin{equation}\label{eq:limit}
	\begin{split}
		u &\in L^2(0,T;W^{1,2}_0(\Om)), \\
		\vr &\in L^\infty(0,T;L^\gamma(\Om)), \\
	\overline{p(\vr)} &\in L^{\frac{\gamma+1}{\gamma}}(0,T;L^{\frac{\gamma+1}{\gamma}}(\Om))
	\end{split}
\end{equation}
and a subsequence $h_j \rightarrow 0$, such that
\begin{equation}\label{eq:conv}
	\begin{split}
			\vr_h &\weakstar \vr \text{ in $L^\infty(0,T;L^\gamma(\Om))$},\\
			u_h &\weak u \text{ in $L^2(0,T;W^{1,2}_0(\Om))$}, \\
			p(\vr_h) &\weak \overline{p(\vr)}
			\text{ in $L^{\frac{\gamma+1}{\gamma}}(0,T;L^{\frac{\gamma+1}{\gamma}}(\Om))$}.
	\end{split}
\end{equation}
Note that we cannot make the identification $\overline{p(\vr)} = p(\vr)$ 
as this would require the density to converge strongly. 
We will prove that this is indeed true in the next section.

To conclude convergence of the product 
terms appearing in the method, we shall need the following 
lemma from \cite{Karlsen2}:
\begin{lemma}\label{lemma:aubin}
Given $T>0$ and a small number $h>0$, write 
$[0,T) = \cup_{k=1}^M[t_{k-1}, t_{k})$ with $t_{k} = hk$ and $Mh = T$. 
Let $\{f_{h}\}_{h>0}^\infty$, $\{g_{h}\}_{h>0}^\infty $ be 
two sequences such that
the mappings $t \mapsto g_{h}(t,x)$ and $t\mapsto f_{h}(t,x)$ 
are constant on each interval $(t_{k-1}, t_{k}]$
and assume that $\{f_{h}\}_{h>0}^\infty$, $\{g_{h}\}_{h>0}^\infty $
converges weakly to
$f$ and $g$ in $L^{p_{1}}(0,T;L^{q_{1}}(\Om))$ and 
$L^{p_{2}}(0,T;L^{q_{2}}(\Om))$, respectively, 
where $1 < p_{1},q_{1}< \infty$ and
$
\frac{1}{p_{1}} + \frac{1}{p_{2}} = \frac{1}{q_{1}} + \frac{1}{q_{2}} = 1.
$
If $\partial_t^k g_h \inb L^1(0,T;W^{-1,1}(\Om))$
and $h^\alpha |f(\cdot, x+h)- f(\cdot, x)| \inb L^{p_2}(0,T;L^{q_2}(\Om))$, for
some $\alpha < 1$, 
then $g_{h}f_{h} \weak gf$ in the sense of distributions on $\Dom$.
\end{lemma}

\begin{lemma}\label{lem:product}
Given the convergences \eqref{eq:conv},
\begin{equation}\label{eq:conv-2}
	\begin{split}
		\vr_h u_h &\weak \vr u \text{ in $L^2(0,T;L^\gamma(\Om))$}, \\
		\vr_h \widehat u_h &\weakstar \vr u
		 	\text{ in $L^\infty(0,T;L^\frac{2\gamma}{\gamma + 1}(\Om))$}, \\
		\vr_h \widehat u_h u_h,~ \vr_h |\widehat u_h|^2,~  \vr_h |u_h|^2 &\weak \vr u^2
		\text{ in $L^1(0,T;L^{\frac{2\gamma}{\gamma+1}}(\Om))$}.
	\end{split}
\end{equation}
\end{lemma}

\begin{proof}
From Corollary \ref{cor:energy} and Lemma \ref{lem:timecont}, we have that \\
$u_h \inb L^2(0,T;W^{1,2}_0(\Om))$ and
$\partial_t^h \vr_h \inb L^2(0,T;W^{-1,\gamma}(\Om))$.
We can then apply Lemma \ref{lemma:aubin}, with 
$g_h = \vr_h$ and $f_h = u_h$, to conclude
$$
	\vr_h u_h \weak \vr u\text{ in $L^2(0,T;L^\gamma(\Om))$}.
$$
Next, we notice that
\begin{equation*}
	\begin{split}
		\left\|\widehat u_h - u_h\right\|_{L^2(\Om)}^2
		&= \sum_i \int_{x_{i-1/2}}^{x_{i+1/2}}
		\left|\frac{1}{\Delta x}\int_{x_{i-1/2}}^{x_{i+1/2}}
		u_h(y)~dy - u_h(x) \right|^2~dx	\\
		&\leq (\Delta x)^2 \|(u_h)_x\|_{L^2(\Om)},
	\end{split}
\end{equation*}
where the last inequality is the Poincar\'e inequality. 
Hence, by writing
\begin{equation}\label{eq:quick}
\vr_h \widehat u_h  = \vr_h u_h + (\vr_h \widehat u_h - u_h),	
\end{equation}
and passing to the limit, we conclude the
second convergence of \eqref{eq:conv-2}.

From Corollary \ref{cor:energy} and Lemma \ref{lem:timecont}, 
we have that \\ $\vr_h|u_h|^2 \inb L^1(0,T;L^\frac{2\gamma}{\gamma+1}(\Om))$
and $\partial_t (\vr_h \widehat u_h) \inb
L^1(0,T;W^{-1,1}(\Om))$. 
Lemma \ref{lemma:aubin} is then applicable, 
with $g_h = \vr_h \widehat u_h$ and $f_h = u_h$, yielding
\begin{equation*}
	\vr_h \widehat u_h u_h
	\weak \vr u^2  \text{ in $L^1(0,T;L^{\frac{2\gamma}{\gamma+1}}(\Om))$}.
\end{equation*}

The remaining convergences 
can be obtained  similarly (i.e \eqref{eq:quick}).
\end{proof}

\subsection{Convergence of the density scheme}
We now prove that the limit $(\vr, u)$ is 
a weak solution of the continuity equation.

\begin{lemma}\label{lem:cont}
Let $(\vr_h, u_h)$ be the numerical approximation 
constructed using Definition \ref{def:scheme}
and \eqref{def:ext1}-\eqref{def:ext2}. 
The limit $(\vr, u)$ is a weak solution 
of the continuity equation. That is,
\begin{equation*}
	\vr_t + (\vr u)_x = 0,
\end{equation*}
in the sense of distributions on $[0,T)\times [0,L]$.
\end{lemma}
\begin{proof}
	{\bf 1.} We first claim that 
	$(\vr_h, u_h)$ satisfies the equation
	\begin{equation}\label{wnum:cont}
		\begin{split}
			\int_0^T\int_0^L \partial_t^h(\vr_h)\phi 
			- \vr_h u_h\,\phi_x~dxdt
			= P_1(\phi), 
		\end{split}
	\end{equation}
	for all $\phi \in C_0^\infty([0,T)\times [0,L])$, where
	$P_1(\phi) = \Delta t \sum_k P_1^k(\phi)$ is given by \eqref{wnum:cont-err}.
	
	 To prove the claim \eqref{wnum:cont}, let $\phi \in C_0^\infty([0,T)\times [0,L])$ be 
	arbitrary and define 
	\begin{equation*}
		\phi_i^k = \frac{1}{\Delta t \Delta x}\int_{k\Delta t}^{(k+1)\Delta t}\int_{x_{i-1/2}}^{x_{i+1/2}}
		\phi(t,x)~dxdt, \quad i=0, \ldots, N-1.
	\end{equation*}
	Now, multiply \eqref{num:cont} with $\phi_i^k\Delta x$, and sum over all $i$, 
	to discover
	\begin{equation}\label{wnum:cont1}
		\begin{split}
			\Delta x\sum_i (\partial_t^k \vr_i)\phi_i^k
			&= \Delta x\sum_i \left(\partial_i \up(\vr^k u^k)\right)\phi_i^k \\
			&=  \sum_i \up(\vr^k u^k)_{i+1/2}\left(\phi^k_i - \phi(\Delta t k, x_{i+1/2})\right) \\
			&\qquad\qquad -\up(\vr^k u^k)_{i-1/2}\left(\phi^k_i - \phi(\Delta t k, x_{i-1/2})\right).
		\end{split}
	\end{equation}
	To proceed, we add and subtract to derive the identities
	\begin{equation}\label{wnum:cont2}
		\begin{split}
			&\up(\vr^k u^k)_{i+1/2}\left(\phi^k_i - \phi(\Delta t k, x_{i+1/2})\right) \\
			&\quad= \vr^k_i u_{i+1/2}^k\left(\phi^k_i - \phi(\Delta t k, x_{i+1/2})\right)\\
			&\qquad \quad+ (\vr^k_{i+1} -\vr_i^k)u_{i+1/2}^-\left(\phi^k_i - \phi(\Delta t k, x_{i+1/2})\right),
		\end{split}
	\end{equation}
	and
	\begin{equation}\label{wnum:cont3}
		\begin{split}
			&\up(\vr^k u^k)_{i-1/2}\left(\phi^k_i - \phi(\Delta t k, x_{i-1/2})\right) \\
			&\quad= \vr^k_i u_{i-1/2}^k\left(\phi^k_i - \phi(\Delta t k, x_{i-1/2})\right)\\
			&\qquad \quad- (\vr^k_{i} -\vr_{i-1}^k)u_{i-1/2}^+\left(\phi^k_i - \phi(\Delta t k, x_{i-1/2})\right).
		\end{split}
	\end{equation}
	By applying \eqref{wnum:cont2}-\eqref{wnum:cont3} in \eqref{wnum:cont1}, we
	obtain
	\begin{equation*}
		\begin{split}
			\Delta x\sum_i (\partial_t^k \vr_i)\phi_i^k
			&= \sum_i\vr^k_i u_{i+1/2}^k\left(\phi^k_i - \phi(\Delta t k, x_{i+1/2})\right) \\
			&\qquad -\vr^k_i u_{i-1/2}^k\left(\phi^k_i - \phi(\Delta t k, x_{i-1/2})\right) \\
			&\qquad +(\vr^k_{i+1} -\vr_i^k)u_{i+1/2}^-\left(\phi^k_i - \phi(\Delta t k, x_{i+1/2})\right) \\
			&\qquad +(\vr^k_{i} -\vr_{i-1}^k)u_{i-1/2}^+\left(\phi^k_i - \phi(\Delta t k, x_{i-1/2})\right).
		\end{split}
	\end{equation*}
	We then apply Green's theorem to obtain
	\begin{equation}\label{wnum:cont4}
		\begin{split}
			&\Delta x\sum_i (\partial_t^k \vr_i)\phi_i^k \\
			&= \sum_i \int_{x_{i-1/2}}^{x_{i+1/2}} \vr_i^k \frac{d}{dx}
			\left(u_h^k(\phi^k_i - \phi(\Delta t k, x))\right)~ dx + P^k_1(\phi),
		\end{split}
	\end{equation}
	where we have defined
	\begin{equation}\label{wnum:cont-err}
		\begin{split}
			P^k_1(\phi) 
			&= \Delta x\sum_i \left[\left(\partial_{i-1/2}\vr^k\right)u_{i-1/2}^++
			\left(\partial_{i+1/2}\vr^k\right)u_{i+1/2}^-\right] \\
			&\qquad \qquad \times\left(\phi^k_i - \phi(\Delta t k, x_{i-1/2})\right).
		\end{split}	
	\end{equation}
	To proceed, we observe the two identities
	\begin{equation}\label{wnum:cont5}
		\Delta t\Delta x\sum_k\sum_i (\partial_t^k \vr_i)\phi_i^k
		= \int_0^T \int_0^L (\partial_t^k \vr_h) \phi~dxdt
	\end{equation}
	and
	\begin{equation}\label{wnum:cont6}
		\begin{split}
			&\Delta t\sum_k\sum_i \int_{x_{i-1/2}}^{x_{i+1/2}} \vr_i^k \frac{d}{dx}
			\left(u_h^k(\phi^k_i - \phi(\Delta t k, x))\right)~ dx \\
			&\quad= \int_0^T\int_0^L \vr_h (u_h)_x (\phi_i - \phi(x)) - \vr_h u_h \phi_x~dx 
			= -\int_0^T\int_0^L \vr_h u_h \phi_x~dx,
		\end{split}
	\end{equation}
	where the last equality follows by definition of $\phi_i^k$
	
	By combining \eqref{wnum:cont4}-\eqref{wnum:cont6}, we obtain
	\eqref{wnum:cont}, which was our claim.
	
	\vspace{0.5cm}
	{\bf2.} Next, we prove that the $P_1(\phi)$ converges 
	to zero as $h\rightarrow 0$. More precisely, we claim that
	\begin{equation}\label{eq:P1}
		\left|P_1(\phi)\right| \leq h^\frac{1}{2}C.
	\end{equation} 
	To prove this, we apply the Cauchy-Schwartz inequality to obtain
	\begin{equation}\label{wnum:P1}
		\begin{split}
			P_1(\phi)
			&=\Delta x\Delta t \sum_k\sum_i \left[\left(\partial_{i-1/2}\vr^k\right)u_{i-1/2}^++
			\left(\partial_{i+1/2}\vr^k\right)u_{i+1/2}^-\right] \\
			&\qquad \qquad \times\left(\phi^k_i - \phi(\Delta t k, x_{i-1/2})\right) \\
			&\leq C\left(\Delta t (\Delta x)^2\sum_k \sum_i
			p''(\vr_\dagger^k)\left|\partial_{i+1/2}\vr_h^k\right|^2|u_{i+1/2}|\right)^\frac{1}{2} \\
			&\qquad \times (\Delta x)^{-1/2}\left(\Delta t \Delta x\sum_k \sum_i
			|u_{i+1/2}|(\vr_\dagger^k)^{2-\gamma} \right.\\
			&\qquad \qquad \qquad \qquad \qquad
			\left.\times \left|\phi^k_i - \phi(\Delta t k, x_{i-1/2})\right|^2\right)^\frac{1}{2}.
		\end{split}
	\end{equation}
	Note that the first term after the inequality is bounded by the energy 
	estimate \eqref{eq:energy}. Now, the H\"older and Poincar\'e inequalities 
	provides the bound
	\begin{equation*}
		\begin{split}
			&(\Delta x)^{-1/2}\left(\Delta t \Delta x\sum_k \sum_i
			|u_{i+1/2}|(\vr_\dagger^k)^{2-\gamma} 
			\left|\phi^k_i - \phi(\Delta t k, x_{i-1/2})\right|^2\right)^\frac{1}{2} \\
			&\leq (\Delta x)^\frac{1}{2}C\|\Grad \phi\|_{L^\infty(0,T;L^\infty(\Om))}
			\|u_h\|_{L^2(0,T;L^\infty(\Om))}^\frac{1}{2}\|\vr_h\|_{L^\infty(0,T;L^\gamma(\Om))} \\
			&\leq (\Delta x)^\frac{1}{2}C\|\Grad \phi\|_{L^\infty(0,T;L^\infty(\Om))},
		\end{split}
	\end{equation*}
	where we have used Corollary \ref{cor:energy} to conclude the last bound. 
	Together with \eqref{wnum:P1}, this proves our claim \eqref{eq:P1}.
	
	\vspace{0.5cm}
	{\bf 3.} Let us now  send $h\rightarrow 0$ in \eqref{wnum:cont} 
	and thereby conclude the proof. For this purpose, 
	we shall need the following elementary identity 
	\begin{equation*}
		\begin{split}
			&\int_0^T\int_0^L \partial_t^k (\vr_h) \phi~dxdt 
			= - \int_0^T\int_0^L \vr_h(t-\Delta t)\partial_t^k \phi~dxdt - \int_0^L \vr_h^0 \phi(\Delta t)~dx,
		\end{split}
	\end{equation*}
	where we have also used that $\phi(T, \cdot) = 0$. With this 
	identity, \eqref{wnum:cont} tell us that
	\begin{equation*}
		\begin{split}
		&- \int_0^T\int_0^L \vr_h(t-\Delta t)\partial_t^k \phi~dxdt
			- \vr_h u_h\,\phi_x~dxdt \\
		&\qquad \qquad	= \int_0^L\vr_h^0 \phi(\Delta t)~dx +   P_1(\phi), \quad \forall \phi \in C_0^\infty([0,T)\times \overline{\Om}).
		\end{split}
	\end{equation*}
	From \eqref{eq:P1}, we have that $P_1(\phi) \rightarrow 0$ as $h \rightarrow 0$. 
	Hence, there is no problem with sending $h \rightarrow 0$, using \eqref{eq:conv} and \eqref{eq:conv-2}, 
	to conclude that
	\begin{equation*}
		\vr_t + (\vr u)_x = 0, 
	\end{equation*}
	in the sense of distributions. This concludes the proof.
\end{proof}

\subsection{Weak limit of the momentum scheme}
In this subsection, we pass to the limit in the momentum scheme 
to conclude that $(\vr, u)$ is almost a weak 
solution of the momentum equation.
We begin by deriving an integral formulation 
of the momentum scheme \eqref{num:moment}.

\begin{lemma}\label{lemma:wnum-mom}
	Let $(\vr_h, u_h)$ be the numerical solution 
	constructed through Definition \ref{def:scheme}
	and \eqref{def:ext1}-\eqref{def:ext2}. Then, 
	for all sufficiently smooth $v$,
\begin{equation}\label{wnum:moment}
	\begin{split}
		&\int_0^t\int_0^L \partial_t^h(\vr_h \widehat u_h)v
		- \vr_h (\widehat u_h)^2  ~v_x~dxdt \\
		&\qquad =\int_0^T\int_0^L (p(\vr_h) - \mu (u_h)_x)v_x~dxdt
		+ P_2(v),
	\end{split}
\end{equation}
where the numerical error term is given by
\begin{equation*}
	\begin{split}
		P_2(v) &=-\Delta t\sum_k\sum_i\int_{x_{i-1/2}}^{x_{i+1/2}}\partial_t^k(\vr_i \widehat u_i)\left(\frac{v_{i-1/2}^k+v_{i+1/2}^k}{2}-v(x)\right)~ dx \\
		&\quad+\frac{\Delta t}{2}\sum_k\sum_i\left(\vr_{i+1}^k\widehat u_{i+1}^k - \vr_{i}^k\widehat u_{i}^k\right)\left[u_{i+1/2}^+(v_{i+3/2} - v_{i+1/2})\right. \\
		&\qquad \qquad\qquad \qquad \qquad \qquad \qquad -\left. u_{i+1/2}^-(v_{i+1/2} - v_{i-1/2})\right].
	\end{split}
\end{equation*}
\end{lemma}
\begin{proof}
Let $v \in C_0^\infty([0,T)\times \Om)$ be arbitrary and 
introduce the notation 
$$
v_{i-1/2}^k = v\left(k\Delta t, x_{i-1/2}\right), \quad i=0,\ldots, N, \quad k=1, \ldots, M .
$$
Now, multiply \eqref{num:moment} with $v_{i+1/2}^k\Delta x$, 
and sum over all $i$ to obtain 
\begin{equation}\label{wnum:mom1}
	\begin{split}
		&\frac{\Delta x}{2}\sum_i \partial_t^k(\vr_i \widehat u_i +  \vr_{i+1/2} \widehat u_{i+1/2})
		v_{i+1/2}^k \\
		&\quad + \frac{1}{2}\sum_i \left(\up\left(\vr^k\widehat u^k u^k\right)_{i+3/2}-\up\left(\vr^k\widehat u^k u^k\right)_{i-1/2}\right)\, v_{i+1/2}^k \\
		&\quad = \Delta x \sum_i \left(\mu\Delta_{i+1/2}u^k - \partial_{i+1/2}p(\vr_h^k)\right)\, v_{i+1/2}^k.
	\end{split}
\end{equation}
We will now write each integral term in \eqref{wnum:mom1} in the form \eqref{wnum:moment}. Let us 
begin with the right-hand side. 

{\bf 1.} Using summation by parts, we calculate
\begin{equation}\label{wnum:mom-2}
	\begin{split}
		&\Delta x \sum_i \left(\mu\Delta_{i+1/2}u^k - \partial_{i+1/2}p(\vr_h^k)\right)\, v_{i+1/2}^k \\
		&\qquad=\sum_i \left(p(\vr_i^k) - \mu \partial_i u^k\right)\left(v_{i+1/2}^k-v_{i-1/2}^k\right) \\
		&\qquad=\sum_i \left(p(\vr_i^k) - \mu \partial_i u^k\right)\int_{x_{i-1/2}}^{x_{i+1/2}}v_x(k\Delta t , x)~dx \\
		&\qquad = \int_0^L (p(\vr^k_h) - \mu (u^k_h)_x)v_x(k\Delta t, x)~dx,
	\end{split}
\end{equation}
which is of the form we wanted. 

{\bf 2.} For the time derivative term, we perform summation by parts to 
deduce 
\begin{equation}\label{wnum:mom-3}
	\begin{split}
		&\frac{\Delta x}{2}\sum_i \partial_t^k(\vr_i \widehat u_i +  \vr_{i+1/2} \widehat u_{i+1/2})
		v_{i+1/2}^k \\
		&\quad = \sum_i\int_{x_{i-1/2}}^{x_{i+1/2}} \partial_t^k(\vr_i \widehat u_i)\left(\frac{v_{i-1/2}^k+v_{i+1/2}^k}{2}\right)~ dx \\
		&\quad =  \int_0^L\partial_t^k(\vr_h \widehat u_h)v~ dx \\
		&\qquad\qquad+ \sum_i\int_{x_{i-1/2}}^{x_{i+1/2}}\partial_t^k(\vr_i \widehat u_i)\left(\frac{v_{i-1/2}^k+v_{i+1/2}^k}{2}-v(x)\right)~ dx \\
		&\quad =: \int_0^L\partial_t^k(\vr_h \widehat u_h)v~ dx + H_1^k
	\end{split}
\end{equation}

{\bf 3.} For the last term, we begin applying summation by parts 
\begin{equation}\label{wnum:mom3}
	\begin{split}
		&\frac{1}{2}\sum_i\left(\up\left(\vr^k\widehat u^k u^k\right)_{i+3/2}-\up\left(\vr^k\widehat u^k u^k\right)_{i-1/2}\right)\, v_{i+1/2}^k \\
		&\qquad= - \frac{1}{2}\sum_i \up(\vr^k \widehat u^k u^k)(v_{i+3/2} - v_{i-1/2}) \\
		&\qquad= - \frac{1}{2}\sum_i \up(\vr^k \widehat u^k u^k)(v_{i+3/2} - v_{i+1/2}) \\
		&\qquad \qquad -\frac{1}{2}\sum_i \up(\vr^k \widehat u^k u^k)(v_{i+1/2} - v_{i-1/2}).
	\end{split}
\end{equation}
Now, by adding and subtracting, we develop the identity
\begin{equation}\label{wnum:mom-id1}
	\begin{split}
		&\up(\vr^k \widehat u^k u^k)(v_{i+3/2} - v_{i+1/2}) \\
		&\qquad= \vr^k_{i+1}\widehat u^k_{i+1} u_{i+1/2}(v_{i+3/2} - v_{i+1/2}) \\
		&\qquad\quad -\left(\vr^k_{i+1}\widehat u^k_{i+1} - \vr^k_i\widehat u^k_{i}\right)u_{i+1/2}^+(v_{i+3/2} - v_{i+1/2}).
	\end{split}
\end{equation}
Similarly, we derive the identity
\begin{equation}\label{wnum:mom-id2}
	\begin{split}
		&\up(\vr^k \widehat u^k u^k)(v_{i+1/2} - v_{i-1/2}) \\
		&\qquad= \vr^k_{i}\widehat u^k_{i} u_{i+1/2}(v_{i+1/2} - v_{i-1/2}) \\
		&\qquad\quad +\left(\vr^k_{i+1}\widehat u^k_{i+1} - \vr^k_i\widehat u^k_{i}\right)u_{i+1/2}^-(v_{i+1/2} - v_{i-1/2}).
	\end{split}
\end{equation}
By applying \eqref{wnum:mom-id1} and \eqref{wnum:mom-id2} in \eqref{wnum:mom3}, we obtain
\begin{equation}\label{wnum:mom-4}
	\begin{split}
		&\frac{1}{2}\sum_i\left(\up\left(\vr^k\widehat u^k u^k\right)_{i+3/2}-\up\left(\vr^k\widehat u^k u^k\right)_{i-1/2}\right)\, v_{i+1/2}^k \\
		&\qquad = -\sum_i\vr_i^k |\widehat u_i^k|^2 \int_{x_{i-1/2}}^{x_{i+1/2}}v_x~dx \\
		&\qquad \quad +\frac{1}{2}\sum_i\left(\vr_{i+1}^k\widehat u_{i+1}^k - \vr_{i}^k\widehat u_{i}^k\right)\left[u_{i+1/2}^+(v_{i+3/2} - v_{i+1/2})\right. \\
		&\qquad \qquad\qquad \qquad \qquad \qquad \qquad -\left. u_{i+1/2}^-(v_{i+1/2} - v_{i-1/2})\right] \\
		&\qquad =: - \int_0^L \vr_h^k \left(\widehat u_h^k\right)^2v_x~ dx + H^k_2.
	\end{split}
\end{equation}

{\bf 4.} By applying \eqref{wnum:mom-2}, \eqref{wnum:mom-3}, and \eqref{wnum:mom-4}, to \eqref{wnum:mom1}, 
multiplying with $\Delta t$, and summing over all $k$,we discover
\begin{equation*}
	\begin{split}
		&\int_0^T\int_0^L\partial_t^k(\vr_h \widehat u_h)v - \vr_h \left(\widehat u_h\right)^2v_x~ dxdt +  \Delta t\sum_k\left(E_1^k+E_2^k\right) \\
		&\qquad =\int_0^T\int_0^L (p(\vr_h) - \mu (u_h)_x)v_x~dxdt,
	\end{split}
\end{equation*}
which is \eqref{wnum:moment} with
\begin{equation*}
	P_2(v) = - \Delta t\sum_k \left(H_1^k+H_2^k\right).
\end{equation*}
This concludes the proof.
\end{proof}

Next, we prove that the error term in \eqref{wnum:moment}
converges to zero as $h \rightarrow 0$.

\begin{lemma}\label{lem:P2}
Let $P_2(v)$, be as in the previous lemma. There 
is a constant $C > 0$, independent of $h$, such that
\begin{equation*}
	\left|P_2(v)\right| \leq h^\frac{1}{4}C\|v_x\|_{L^\infty(0,T;L^\infty(\Om))},
\end{equation*}
for all $v \in L^\infty(0,T;W^{1, \infty}_0(\Om))$.
\end{lemma}
\begin{proof}
By definition, we have that
	\begin{equation}\label{eq:P2}
		\begin{split}
			P_2(v) &=-\Delta t\sum_k\sum_i\int_{x_{i-1/2}}^{x_{i+1/2}}\partial_t^k(\vr_i \widehat u_i)\left(\frac{v_{i-1/2}^k+v_{i+1/2}^k}{2}-v(x)\right)~ dx \\
			&\quad+\frac{\Delta t}{2}\sum_k\sum_i\left(\vr_{i+1}^k\widehat u_{i+1}^k - \vr_{i}^k\widehat u_{i}^k\right)\left[u_{i+1/2}^+(v_{i+3/2} - v_{i+1/2})\right. \\
			&\qquad \qquad\qquad \qquad \qquad \qquad \qquad -\left. u_{i+1/2}^-(v_{i+1/2} - v_{i-1/2})\right] \\
			&=: S_1 + S_2.
		\end{split}
	\end{equation}	
Let us now bound the $S_1$ and $S_2$ term separately. 

\vspace{0.5cm}
{\it {\bf 1.} Bound on $S_1$:} By adding and subtracting, we write
\begin{equation}\label{wnum:S1}
	\begin{split}
		S_1 &= -\Delta t\sum_k\sum_i\int_{x_{i-1/2}}^{x_{i+1/2}}\partial_t^k(\vr_i \widehat u_i)\left(\frac{v_{i-1/2}^k+v_{i+1/2}^k}{2}-v(x)\right)~ dx \\
		&= -\Delta t\sum_k\sum_i\int_{x_{i-1/2}}^{x_{i+1/2}}\vr^{k-1}_i \partial_t^k(\widehat u_i)\left(\frac{v_{i-1/2}^k+v_{i+1/2}^k}{2}-v(x)\right)~ dx \\
		&\qquad -\Delta t\sum_k\sum_i\int_{x_{i-1/2}}^{x_{i+1/2}}\widehat u^{k}_i \partial_t^k(\vr_i)\left(\frac{v_{i-1/2}^k+v_{i+1/2}^k}{2}-v(x)\right)~ dx \\
		&=: K_1 + K_2.
	\end{split}
\end{equation}
To bound the $K_1$ term, we apply the Cauchy-Schwartz and
Poincar\'e inequalities to obtain
\begin{align}
		\left|K_1\right| &=\left|\Delta t\sum_k\sum_i\int_{x_{i-1/2}}^{x_{i+1/2}}\vr^{k-1}_i \partial_t^k(\widehat u_i)\left(\frac{v_{i-1/2}^k+v_{i+1/2}^k}{2}-v(x)\right)~ dx\right|\nonumber\\
		&\leq (\Delta t)^{-\frac{1}{2}}\left((\Delta t)^2 \Delta x \sum_{k}\sum_i \vr_i^{k-1}\left|\partial_t^k \widehat u_i^k\right|^2\right)^\frac{1}{2} \label{wnum:K1}\\
		&\qquad \qquad \times \left(\Delta t\Delta x\sum_k\sum_i \vr_h^{k-1}\left|\left(\frac{v_{i-1/2}^k+v_{i+1/2}^k}{2}-v(x)\right)~ dx\right|^2\right)^\frac{1}{2} \nonumber\\
		&\leq (\Delta t)^{-\frac{1}{2}}C \mathcal{M}^\frac{1}{2} (\Delta x)\|v_x\|_{L^\infty(0,T;L^2(\Om))} 
		\leq h^\frac{1}{2}C\|v_x\|_{L^\infty(0,T;L^\infty(\Om))},\nonumber
\end{align}
where we have used that the term involving $\partial_t^k \widehat u_h$ is bounded 
by the energy estimate \eqref{eq:energy}.

To bound the $K_2$ term, we apply yet another application 
of the Cauchy-Schwartz and Poincar\'e inequalities to discover
\begin{align}
		|K_2| &= \left|\Delta t\sum_k\sum_i\int_{x_{i-1/2}}^{x_{i+1/2}}\widehat u^{k}_i \partial_t^k(\vr_i)\left(\frac{v_{i-1/2}^k+v_{i+1/2}^k}{2}-v(x)\right)~ dx \right| \nonumber\\
		&\leq (\Delta t)^{-\frac{1}{2}}\left((\Delta t)^2 \Delta x\sum_{k}\sum_i p''(\vr_i^\ddagger)\left|\partial_t^k \vr_h\right|^2\right)^\frac{1}{2}\label{wnum:K2}\\
		&\quad \quad \times \left(\Delta t \Delta x\sum_{k}\sum_i (\vr_i^\ddagger)^{2-\gamma}|\widehat u_i^k|^2\left|\left(\frac{v_{i-1/2}^k+v_{i+1/2}^k}{2}-v(x)\right)~ dx\right|^2 \right)^\frac{1}{2}\nonumber\\
		&\leq (\Delta t)^{-\frac{1}{2}}C\|u_h\|_{L^2(0,T;L^\infty(\Om))}\mathcal{M}^\frac{1}{2}(\Delta x)\|v_x\|_{L^\infty(0,T;L^\infty(\Om))}. \nonumber
\end{align}

Setting \eqref{wnum:K1}-\eqref{wnum:K2} in \eqref{wnum:S1} yields
\begin{equation}\label{eq:S1}
	S_1 \leq h^\frac{1}{2}C\|v_x\|_{L^\infty(0,T;L^2\infty(\Om))}.
\end{equation}

\vspace{0.5cm}
{\it {\bf 2.} Bound on $S_2$:}
By adding and subtracting, we calculate
\begin{equation}\label{wnum:S2}
	\begin{split}
		|S_2| &= \left|\frac{\Delta t}{2}\sum_k\sum_i\left(\vr_{i+1}^k\widehat u_{i+1}^k - \vr_{i}^k\widehat u_{i}^k\right)\left[u_{i+1/2}^+(v_{i+3/2} - v_{i+1/2})\right.\right. \\
		&\qquad \qquad\qquad \qquad \qquad \qquad \qquad \left.-\left. u_{i+1/2}^-(v_{i+1/2} - v_{i-1/2})\right]\right| \\
		&\leq (\Delta x)\|v_x\|_{L^\infty(0,T;L^\infty)}
		 \left(\Delta x\Delta t\sum_k\sum_i \left|\partial_{i+1/2}\vr_h\right||u_{i+1/2}^k||\widehat u_i^k| \right. \\
		&\qquad \qquad\qquad \qquad\qquad \qquad \left. + \left|\partial_{i+1/2}u_h^k\right||u_{i+1/2}|\vr_{i+1}^k\right) \\
		&=: (\Delta x)\|v_x\|_{L^\infty(0,T;L^\infty)}\left(K_3 + K_4\right).
	\end{split}
\end{equation}
To bound the $K_3$ term, we apply the Cauchy-Schwartz inequality, 
followed by the H\" older inequality, and obtain
\begin{equation}\label{wnum:K3}
	\begin{split}
		K_3 &=\Delta t\Delta x\sum_k\sum_i \left|\partial_{i+1/2}\vr_h\right||u_{i+1/2}^k||\widehat u_i^k| \\
		&\leq (\Delta t)^{-\frac{1}{2}}\left((\Delta t)^2 \Delta x\sum_k \sum_ip''(\vr_\dagger^k)\left|\partial_{i+1/2}\vr_h\right|^2|u_{i+1/2}^k|\right)^\frac{1}{2} \\
		&\qquad \qquad \times \left(\Delta t \Delta x \sum_k\sum_i |u_{i+1/2}^k| \left|\widehat u_i^2\right|^2 (\vr_\dagger^k)^{2-\gamma}\right)^{\frac{1}{2}}\\
		&\leq (\Delta t)^{-\frac{1}{2}}C \mathcal{M}^\frac{1}{2}\|u_h\|_{L^3(0,T;L^\infty(\Om))}^\frac{3}{2} \\
		&\leq (\Delta t)^{-\frac{1}{2}}C\, (\Delta t)^{\frac{3}{2}\left(\frac{1}{3} - \frac{1}{2}\right)}\|u_h\|_{L^2(0,T;L^\infty(\Om))}^\frac{3}{2} 
		\leq (\Delta t)^{-\frac{3}{4}}C,
	\end{split}
\end{equation}
where we have also utilized a standard inverse estimate in time and the energy estimate \eqref{eq:energy}.

Next, we apply the H\"older inequality to deduce
\begin{equation}\label{wnum:K4}
	\begin{split}
		K_4 &= \Delta t\Delta x\sum_k\sum_i\left|\partial_{i+1/2}u_h^k\right||u_{i+1/2}|\vr_{i+1}^k \\
		&\leq \|(u_h)_x\|_{L^2(0,T;L^2(\Om))}\|u_h\|_{L^2(0,T;L^\infty(\Om))}\|\vr_h\|_{L^\infty(0,T;L^2(\Om))} \\
		&\leq (\Delta x)^{\frac{1}{2}- \frac{1}{\gamma}}\|\vr_h\|_{L^\infty(0,T;L^\gamma(\Om))}C,
	\end{split}
\end{equation}
where we have utilized an inverse estimate and the energy estimate \eqref{eq:energy}.

By setting \eqref{wnum:K3} and \eqref{wnum:K4} in \eqref{wnum:S2}, we conclude
\begin{equation}\label{eq:S2}
	|S_2| \leq h^\frac{1}{4}\|v_x\|_{L^\infty(0,T;L^\infty(\Om))}.
\end{equation}

Finally, we apply \eqref{eq:S1} and \eqref{eq:S2} in \eqref{eq:P2}
to conclude the proof.

\end{proof}

\begin{lemma}\label{lem:moment}
Let $(\vr_h, u_h)$ be the numerical approximation 
constructed using Definition \ref{def:scheme}
and \eqref{def:ext1}-\eqref{def:ext2}. 
The limit $(\vr, u, \overline{p(\vr)})$ satisfies 
\begin{equation}\label{wcont:moment}
	(\vr u)_t + (\vr u^2)_x = \mu u_{xx} - \overline{p(\vr)}_x,
\end{equation}
in the sense of distributions on $[0,T)\times \Om$.
\end{lemma}
\begin{proof}
We begin by writing \eqref{wnum:moment} in the form
\begin{equation}\label{wnum:moment2}
	\begin{split}
		&-\int_0^T\int_0^L (\vr_h \widehat u_h)(-\Delta t, \cdot)\, \partial_t^h v
		+ \vr_h (\widehat u_h)^2  ~v_x~dxdt \\
		& =\int_0^T\int_0^L (p(\vr_h) - \mu (u_h)_x)v_x~dxdt
		+\int_0^L (\vr_h^0\widehat u_h^0)~v(\Delta, \cdot)~dx
		+ P_2(v).
	\end{split}
\end{equation}
The convergence \eqref{eq:conv} provides $p(\vr_h) \weak \overline{p(\vr)}$
and from Lemma \ref{lem:P2}, we have that $P_2(v)\rightarrow 0$
as $h \rightarrow 0$. Moreover, from Lemma \ref{lem:product}, 
we know that $\vr_h \widehat u_h \weakstar \vr u$ in $L^\infty(0,T;L^{2\gamma}{\gamma+1}(\Om))$
and  $\vr_h |\widehat u_h|^2 \weak \vr u^2$ in $L^1(0,T;L^\frac{2\gamma}{\gamma+1}(\Om))$.
Hence, there is no problems with sending $h \rightarrow 0$ in \eqref{wnum:moment2}
to conclude the proof.
\end{proof}

\subsection{The effective viscous flux limit}
We are going to end this section by passing to the limit
in the effective viscous flux equation \eqref{eq:eff}.
To achieve this, we shall need the following 
result.
\begin{lemma}\label{lem:}
Let $\{f_h\}_{h>0}$, $\{g_h\}_{h>0}$ be two sequences satisfying
\begin{itemize}
	\item For each fixed $h>0$, $f_h$ and $g_h$ are piecewise 
	linear and piecewise constant, respectively, with respect to our grid with $\Delta t = \Delta x = h$:
	\begin{equation*}
	\qquad \quad g_h(t,x) = g_i^k, \quad v_h(t,x) = v_{i-1/2}^k + \frac{x-x_{i-1/2}}{\Delta x}(v_{i+1/2} - v_{i-1/2}),	
	\end{equation*}
	for $(t,x) \in [k\Delta t, (k+1)\Delta t)\times (x_{i-1/2}, x_{i+1/2})$, $\forall k$, $i$.
	
	\item As $h \rightarrow 0$, $f_h \weakstar f$ and $g_h \weakstar g$ 
	in $L^{\infty}(0,T;L^{p}(0,L))$ and \\ 
	$L^\infty(0,T;(0,L))$, respectively, 
	where $1/p + 1/q \leq 1$.
	
	\item The discrete time derivate of $f_h$ satisfies
	$$
		\partial_t^k v_h \inb L^1(0,T;W^{-1,1}(0,L)).
	$$
\end{itemize}
Then, for any $t \in (0,T)$, 
	\begin{equation}\label{jada}
		\lim_{h \rightarrow 0} \left(\Delta x\sum_i
		\partial_{i}\Delta^{-1}_{i+1/2}\left[v^k_i\right]g_i^k\right)
		 = \int_0^L \partial_x\Delta_D^{-1}[f(t)]\, g(t)~dx,
	\end{equation}
	where $k$ is given by $k = \lfloor t/h \rfloor$ and $\Delta_D$ denotes the Laplace
	operator with homogenous Dirichlet conditions.
\end{lemma}
\begin{proof}
Let $q_h$ be the extension of $\partial_{i}\Delta^{-1}_{i+1/2}\left[v^k_{i+1/2}\right]$
to all of $[0,T)\times (0,L)$;
\begin{equation*}
	q_h(t,x) = \partial_{i}\Delta^{-1}_{i+1/2}\left[v^k_{i+1/2}\right], \quad 
\end{equation*}
$\forall (t,x) \in [k\Delta t, (k+1)\Delta t)\times (x_{i-1/2}, x_{i+1/2})$
and relevant $i,k$. From the requirements on $f_h$, it is clear that
$$
 \partial_{i+1/2}q_h, \quad \partial_t^k q_h \in_b L^1(0,T;L^1(0,L)),
$$
Then, we can apply Lemma \ref{lemma:aubin}, 
with $f_h = g_h = q_h$
to conclude 
that $q_h^2 \weak q^2$ in the sense of distributions. Thus, 
$q_h \rightarrow q$ a.e as $h \rightarrow 0$, 
where the convergence may take place along a subsequence. Observe
that the limit must satisfy
$$
	q = \partial_x \Delta_D^{-1}[v],
$$
where $\Delta_D$ is the Laplacian with Dirichlet conditions.

Next, let $q_h^L$ be the linear-in-time interpolation 
of $q_h$:
\begin{equation*}
	q^L_h(t, x) = q_h^k(x) 
	+ \frac{t-k\Delta t}{\Delta t}\left(q_h^{k+1}(x)-q_h^k(x)\right), \quad 
	t \in [k\Delta t, (k+1)\Delta t),
\end{equation*}
for $k= 0, \ldots, M-1$. Since 
we have that
\begin{equation*}
	\frac{\partial q^L_h(t,x)}{\partial t} = \partial_t^h q_h
	\inb L^1(0,T;L^1(0,L)),
\end{equation*}
we  have that $q_h^L \rightarrow q$ in $C(0,T;L^p(0,L))$
and $\sup_t\|q_h^L - q_h\|_{L^p(\Om)} \rightarrow 0$,
for any $p < \infty$. Finally, we write
\begin{equation*}
	\begin{split}
		&\Delta x\sum_i
		\partial_{i}\Delta^{-1}_{i+1/2}\left[v^k_i\right]g_i^k =
		\int_0^L q_h^L g_h~dx
		+ \int_0^L (q_h^L - q_h)g_h~dx,
	\end{split}
\end{equation*}
and send $h \rightarrow 0$ to discover \eqref{jada}.
\end{proof}

Equipped with the previous lemma, 
we now pass to the limit in the 
effective viscous flux equation \eqref{eq:eff}.

\begin{lemma}\label{lem:eff-cont}
Assume that $\gamma > \frac{3}{2}$ and
that the convergences \eqref{eq:conv} 
and \eqref{eq:conv-2} holds. Then,
for any $t \in (0,T)$,
\begin{equation}\label{eq:eff-cont}
	\begin{split}
		&-\int_0^t\int_0^L \overline{\left(\mu u_x - p\left(\vr\right)\right)\left(\vr - \frac{\mathcal{M}}{L}\right)}~dxdt \\
		&\qquad = \int_0^t\int_0^L \vr u^2~dxdt - \left.\int_0^L\partial_{x}\Delta^{-1}_D\left[\vr u\right]\vr~dxdt\right|_{s=0}^t,
	\end{split}
\end{equation}	
where the overline denotes the weak $L^1$ limit.
\end{lemma}
\begin{proof}
Let $m = \lfloor t/\Delta t\rfloor$, then \eqref{eq:eff} 
provides the identity
\begin{equation}\label{eq:eff-start}
	\begin{split}
		&-\int_0^{m\Delta t}\int_0^L \left(\mu (u_h)_x - p\left(\vr_h\right)\right)\left(\vr_h - \frac{\mathcal{M}}{L}\right)~dxdt \\
		&\qquad =  - \Delta x \sum_i\partial_{i}\Delta^{-1}_{i+1/2}\left[\frac{\vr^{m}_i \widehat u^{m}_i + \vr^{m}_{i+1} \widehat u^{m}_{i+1}}{2}\right]
		\vr^M_i\\
		&\qquad \quad +\Delta t\Delta x \sum_{k=0}^m\sum_i \up(\vr^k \widehat u^k u^k)_{i+1/2}\left(\frac{\mathcal{M}}{L}\right)\\
		&\qquad \quad + \Delta x \sum_i\partial_{i}\Delta^{-1}_{i+1/2}\left[\frac{\vr^{0}_i \widehat u^{0}_i + \vr^{0}_{i+1} \widehat u^{0}_{i+1}}{2}\right]
		\vr^1_i + E_1^m + E_2^m.
	\end{split}
\end{equation}	
By virtue of the previous lemma, we know that the product terms involving 
$\partial_{i}\Delta^{-1}_{i+1/2}[\cdot]$ converges
to the product of the limits. Furthermore,
from Lemmas \ref{lem:E1} and \ref{lem:E2}, we know
that the $E^m_1$ and $E^m_2$ terms converges to zero 
as $h \rightarrow 0$. Hence, the only term that might
be problematic is the term involving $\up(\vr \widehat u u)$.

By adding and subtracting, we write
\begin{equation*}
	\begin{split}
		&\Delta t\Delta x \sum_{k=0}^m\sum_i \up(\vr^k \widehat u^k u^k)_{i+1/2}\left(\frac{\mathcal{M}}{L}\right) \\
		&\qquad = \left(\frac{\mathcal{M}}{L}\right) \Delta t \Delta x
		\sum_{k=0}^M \sum_i (\vr_i^k \widehat u_i^k) u_{i+1/2}^{k,+} + (\vr_{i+1}^k\widehat u_{i+1}^k)u_{i+1/2}^{k,-} \\
		&\qquad = \left(\frac{\mathcal{M}}{L}\right) \Delta t \Delta x
		\sum_{k=0}^M \sum_i (\vr_i^k\widehat u_i^k) u^k_{i+1/2} \\
		&\qquad = \left(\frac{\mathcal{M}}{L}\right) \int_0^T\int_0^L \vr_h |\widehat u_h|^2~dxdt \\
		&\qquad\qquad\qquad+\left(\frac{\mathcal{M}}{L}\right) \Delta t \Delta x
		\frac{1}{2}\sum_{k=0}^M \sum_i (\vr_i^k\widehat u_i^k) (u^k_{i+1/2}-u^k_{i-1/2}),
	\end{split}
\end{equation*}
where the last term is bounded as
\begin{equation*}
	\begin{split}
	  &\left|\Delta t \Delta x\frac{1}{2}\sum_{k=0}^M \sum_i (\vr_i^k\widehat u_i^k) (u^k_{i+1/2}-u^k_{i-1/2})\right|	\\
	&\qquad \leq (\Delta x)\|\vr_hu_h\|_{L^2(0,T;L^2(\Om))}
	\|(u_h)_x\|_{L^2(0,T;L^2(\Om))} \\
	&\qquad \leq (\Delta x)^{1+\frac{1}{2}-\frac{1}{\gamma}}
	\|\vr_hu_h\|_{L^2(0,T;L^\gamma(\Om))}C, 
	\end{split}
\end{equation*}
where the last inequality is a standard inverse estimate.
The two previous equations tell us that 
\begin{equation*}
	\begin{split}
		&\lim_{h \rightarrow 0}\Delta t\Delta x \sum_{k=0}^m\sum_i \up(\vr^k \widehat u^k u^k)_{i+1/2}
		=\int_0^T\int_0^L \vr u^2~dxdt.
	\end{split}
\end{equation*}
Hence, there is no problems 
with passing to the limit $h \rightarrow 0$ in \eqref{eq:eff-start}
and conclude the proof.
\end{proof}

\section{Strong convergence (proof of Theorem \ref{thm:main})}
In the previous section, we proved that the 
method almost converges  to  a weak solution 
of the compressible Navier-Stokes equations. To 
conclude convergence of the method, and thereby 
conclude the main theorem, it only remains to prove that
$$
	\overline{p(\vr)} = \vr \quad \text{a.e}.
$$
This is the topic of this section. To make 
this identification, we shall prove that $\vr_h$ converges strongly a.e. 
Our strategy will be to adopt the continuous arguments of Lions \cite{Lions:1998ga} to the numerical setting.
For this purpose, we shall need the following well-known lemma.
The reader may consult \cite{Feireisl} for a proof.

\begin{lemma}\label{lem:}
Let $(\vr, u)$ satisfy the continuity equation \eqref{eq:cont} in 
the sense of distributions. If $\vr \in L^2(0,T;L^2(\Om))$ and $u \in L^2(0,T;W^{1,2}_0(0,L))$, then 
for all $B \in C^1(\R^+)$ such that $B(\vr) \in L^1(0,T;L^1(\Om))$,
\begin{equation*}\label{eq:renorm}
	B(\vr)_t + (B(\vr)u)_x + b(\vr)u_x = 0, \qquad b(\vr) = \vr B'(\vr)- B(\vr),	
\end{equation*}
in the sense of distributions on $[0,T)\times [0,L]$.
\end{lemma}

From Lemma \ref{lem:cont}, we know that the limit $(\vr, u)$
is a weak solution of the continuity equation and moreover that
$\vr \in L^2(0,T;L^2(\Om))$ and $u \in L^2(0,T;W^{1,2}_0(0,L))$. 
The previous lemma is then applicable.  By setting $B(z) = z \log z$
and integrating over the domain, we obtain
\begin{equation}\label{eq:strong1}
	\int_0^L \vr \log \vr~dx~(t) = \int_0^L \vr_0 \log \vr_0~ dx -\int_0^t\int_0^L \vr u_x~dx,
\end{equation}	
for any $t \in (0,T)$.

Next, we set $B(z) = z\log z$ in the renormalized continuity scheme \eqref{num:renorm}, multiply
with $\Delta x$, and sum over $k=1, \ldots, m$, to obtain, for any $m =1, \ldots, M$,
\begin{equation}\label{eq:strong2}
	\int_0^L \vr_h^m \log \vr_h^m~dx \leq \int_0^L \vr_h^0 \log \vr_h^0~dx- \int_0^{m\Delta t}\int_0^L \vr_h (u_h)_x~dxdt,
\end{equation}
where we have also used convexity of the map $z \mapsto B(z)$ to conclude a sign on 
the error terms.

By subtracting \eqref{eq:strong1} from \eqref{eq:strong2} and 
passing to the limit $h \rightarrow 0$, we deduce
\begin{equation}\label{strong:1}
	\begin{split}
		0 \leq \left(\int_0^L \overline{\vr \log \vr} - \vr \log \vr~dx\right)(t)
		\leq \int_0^t \int_0^L \vr u_x - \overline{\vr u_x}~dxdt, 
	\end{split}
\end{equation}
for any $t \in (0,T)$.
Here, the overbar denotes weak $L^1$-limit and the first 
inequality is a consequence of convexity. 

The remaining ingredient to prove strong convergence 
of the density is to prove that the right-hand side in \eqref{strong:1} 
is negative:

\begin{lemma}\label{lem:}
Let $(\vr_h, u_h)$ be the numerical solution constructed 
using Definition \ref{def:scheme} and \eqref{def:ext1}-\eqref{def:ext2}.
Then, 
\begin{equation}\label{itisneg}
	\int_0^t \int_0^L \vr u_x - \overline{\vr u_x}~dxdt \leq 0.
\end{equation}
As a consequence, $\vr_h \rightarrow \vr$ a.e as $h \rightarrow 0$.
\end{lemma}
\begin{proof}
Let $\Delta_N$ denote the Laplace operator with homogenous Neumann
conditions on $(0,L)$.
From \eqref{eq:limit} and Lemma \ref{lem:timecont}, we have that $\vr \in L^\infty(0,T;L^\gamma(\Om))$
and $\vr_t \in L^2(0,T;W^{-1,\gamma}(\Om))$, respectively. 
As a consequence,  
\begin{equation*}
	\begin{split}
			v = \partial_x \Delta_N^{-1}\left[\vr- \frac{\mathcal{M}}{L}\right] = \int_0^x \vr - \frac{\mathcal{M}}{L}~ dx,
	\end{split}
\end{equation*}
satisfies $v_x$, $v_t \in L^2(0,T;L^\gamma(\Om))$. 
Hence, we can set $v$ as test-function 
in \eqref{wcont:moment} to obtain
\begin{equation}\label{final:1}
	\begin{split}
		&\int_0^t\int_0^L \left(\mu u_x - \overline{p(\vr)}\right)\left(\vr - \frac{\mathcal{M}}{L}\right)~dxdt \\
		&\qquad = -\int_0^t \int_0^L (\vr u)\partial_t \partial_x \Delta^{-1}_N[\vr] + \vr u^2\left(\vr - \frac{\mathcal{M}}{L}\right)~dxdt \\
		&\qquad \qquad \qquad -  \int_0^L \partial_x \Delta^{-1}_D[\vr_0 u_0]\vr_0~dx+ \left(\int_0^L \partial_x \Delta^{-1}_D[\vr u]\vr~dx\right)(t) \\
		&\qquad = -\int_0^t \int_0^L \partial_t \left(\partial_x \Delta^{-1}_D[\vr u]\right) \vr+ \vr u^2\left(\vr - \frac{\mathcal{M}}{L}\right)~dxdt \\
		&\qquad \quad 
	\end{split}
\end{equation}
where the last equality is integration by parts and the duality of the
 Neumann and Dirichlet Laplacians $\Delta_N$ and $\Delta_D$, respectively. 

To proceed, we note that $\phi = \partial_x \Delta^{-1}_D[\vr u]$ is a valid test-function 
for the continuity equation \eqref{eq:cont}. Since $(\vr, u)$ is a weak solution 
of the continuity equation (Lemma \ref{lem:cont}), we deduce the identity
\begin{equation}\label{final:2}
	\begin{split}
		&\int_0^t \int_0^L \partial_t \left(\partial_x \Delta^{-1}_D[\vr u]\right) \vr~dxdt \\
		&\qquad = -\int_0^t \int_0^L \vr u \partial_x^2 \Delta^{-1}_D[\vr u]~dxdt - \left. \int_0^L \partial_x \Delta^{-1}_D[\vr_0 u_0] \vr_0~dx\right|_{t=0}^t \\
		&\qquad = -\int_0^t \int_0^L (\vr u)^2~dxdt  - \left. \int_0^L \partial_x \Delta^{-1}_D[\vr_0 u_0] \vr_0~dx\right|_{t=0}^t \\
	\end{split}
\end{equation}

By setting \eqref{final:2} in \eqref{final:1}, we discover the identity
\begin{equation}\label{final:3}
	\begin{split}
		&\int_0^t\int_0^L \left(\mu u_x - \overline{p(\vr)}\right)\left(\vr - \frac{\mathcal{M}}{L}\right)~dxdt \\
		&\qquad =\left(\frac{\mathcal{M}}{L}\right)\int_0^t \int_0^L \vr u^2 ~ dxdt +  \left. \int_0^L \partial_x \Delta^{-1}_D[\vr_0 u_0] \vr_0~dx\right|_{t=0}^t \\
		&\qquad =\int_0^t\int_0^L \overline{\left(\mu u_x - p\left(\vr\right)\right)\left(\vr - \frac{\mathcal{M}}{L}\right)}~dxdt,
	\end{split}
\end{equation}
where the last equality is \eqref{eq:eff-cont}.

Finally, we rewrite \eqref{final:3} in the form
\begin{equation}\label{final:final}
	\begin{split}
		\int_0^t\int_0^L \mu u_x \vr - \mu \overline{u_x \vr}~ dxdt
		= \int_0^t\int_0^L \overline{p(\vr)}\vr - \overline{p(\vr)\vr}~dxdt \leq 0,
	\end{split}
\end{equation}
where the last inequality follows from the convexity of $z \mapsto p(z)$.
This concludes the proof of \eqref{itisneg}.
	
By combining \eqref{final:final} and \eqref{strong:1}, we conclude that
$$
 \overline{\vr \log \vr} = \vr \log \vr \text{ a.e on $(0,T)\times (0,L)$},
$$
and hence $\vr_h \rightarrow \vr$ a.e.
\end{proof}

\subsection*{Proof of Theorem \ref{thm:main}:}
From Lemma \ref{lem:moment}, we have that $(\vr, u)$ 
solves
$$
(\vr u)_t + (\vr u^2)_x = \mu u_{xx} - \overline{p(\vr)}_x,
$$
in the sense of distributions. The previous lemma 
tell us that 
$$
\overline{p(\vr)} = p(\vr) \text{ a.e in $(0,T)\times (0,L)$}.
$$
Hence, $(\vr, u)$ is a weak solution of the compressible 
Navier-Stokes system \eqref{eq:cont}-\eqref{eq:moment}. 
This concludes the proof of our main theorem (Theorem \ref{thm:main}).
\qed

\end{document}